\documentclass{amsart}
\usepackage{amsmath,amsthm} 
\usepackage{graphicx}

\advance\hoffset by -0.9 truecm
\begin{document} 
\newcommand{\s}{\vspace{0.2cm}} 
\newcommand{\bSalg}{\bar{\mathcal S}_{alg}} 
\newcommand{\bStop}{\bar{\mathcal S}_{top}} 
\newcommand{\bSun}{\bar{\mathcal S}_{un}} 
\newcommand{\bbb}{\mathbb} 
\newcommand{\Ch}{\widehat{\mathbb C}} 
\newcommand{\Hh}{{\mathbb H}^3} 
\newcommand{\Hw}{{\mathbb H}^2} 
\newcommand{\M}{\mathcal M} 
\newcommand{\Mzp}{{\mathcal M}_{0,2p}} 
\newcommand{\Nalg}{{\mathcal N}_{alg}} 
\newcommand{\ngn}{^{-1}} 
\newcommand{\noi}{\noindent} 
\newcommand{\Ntop}{{\mathcal N}_{top}} 
\newcommand{\Nun}{{\mathcal N}_{un}} 
\newcommand{\PGLR}{PGL(2,{\mathbb R})} 
\newcommand{\PPGLR}{P^+GL(2,{\mathbb R})} 
\newcommand{\PSLC}{PSL(2,{\mathbb C})} 
\newcommand{\PSLR}{PSL(2,{\mathbb R})} 
\newcommand{\R}{\mathcal R} 
\newcommand{\Rc}{{\mathcal R}^c} 
\newcommand{\Salg}{{\mathcal S}_{alg}} 
\newcommand{\Salgc}{{\mathcal S}_{alg}^c} 
\newcommand{\SLR}{SL(2,{\mathbb R})} 
\newcommand{\SPMLR}{S^{\pm}L(2,{\mathbb R})} 
\newcommand{\Stab}{\mathop{\rm Stab}\nolimits} 
\newcommand{\Stop}{{\mathcal S}_{top}} 
\newcommand{\Sun}{{\mathcal S}_{un}} 
\newcommand{\Salgn}{{\mathcal S}_{alg}^n} 
\newcommand{\Salgo}{{\mathcal S}_{alg}^0} 
\newcommand{\Stopc}{{\mathcal S}_{top}^c} 
\newcommand{\Stopn}{{\mathcal S}_{top}^n} 
\newcommand{\Stopo}{{\mathcal S}_{top}^0} 
\newcommand{\Suno}{{\mathcal S}_{un}^0} 
\newcommand{\Sunc}{{\mathcal S}_{un}^c} 
\newcommand{\Sunn}{{\mathcal S}_{un}^n} 
\newcommand{\T}{\mathcal T} 
\newtheorem{Thm}{Theorem}[section] 
\newtheorem{Prop}{Proposition}[section] 
\newtheorem{Cor}{Corollary}[section] 
\newtheorem{Lem}{Lemma}[section] 
\newtheorem{Remark}{Remark}[section]

\title[On neoclassical Schottky groups]{On neoclassical Schottky groups} 
\author{Rub\'en Hidalgo}
\address{Departamento de Matem\'atica, UTFSM, Valpara\'{\i}so, Chile}
\email{ruben.hidalgo@usm.cl}
\author{Bernard Maskit}
\address{Math. Departament, SUNY at Stony Brook, N.Y., USA}
\email{bernie@math.sunysb.edu} 
\thanks{Partially supported by Projects Fondecyt 1030252, 1030373, 7000715 and UTFSM 12.03.21} 

%%%%%%%%%%%%%%%%%%
%%%%%%%%%%%%%%%%%%
\begin{abstract} The goal of this paper is to describe a theoretical 
construction of an infinite collection of non-classical Schottky groups. 
We first show that there are infinitely many non-classical 
noded Schottky groups on the boundary of Schottky space, and we show that 
infinitely many of these are ``sufficiently complicated''. 
We then show that every Schottky group in an appropriately defined relative conical neighborhood 
of any sufficiently complicated noded Schottky group is necessarily non-classical. 
Finally, we construct two examples; the first is a noded Riemann surface of 
genus $3$ that cannot be uniformized by any neoclassical Schottky group (i.e., 
classical noded Schottky group); the second is an explicit example of a 
sufficiently complicated noded Schottky group in genus $3$. \end{abstract}
\maketitle 
\begin{center} 
%\today 
\end{center}

%%%%%%%%%%%%%%
%%%%%%%%%%%%%%
\section*{Introduction}
It was shown by Marden \cite{marden:schottky} that there are Schottky groups
that are not classical (that is, they cannot be defined by circles); see also
\cite{J-M-M:schottky}. An explicit family of examples of non-classical Schottky groups was
constructed by Yamamoto \cite{yamamoto:example}. The main goal of this paper is
to give a theoretical construction of infinitely many non-classical Schottky
groups. A geometrically finite Kleinian group on the boundary of Schottky space
(i.e., a geometrically finite free Kleinian group that contains parabolic
elements) is a {\em noded Schottky group}; the noded Schottky group is {\em
neoclassical} if it can be defined by Euclidean circles. We show that, for each
genus $\ge 2$, there are infinitely many topologically distinct noded Schottky
groups; that only finitely many of these topologically distinct noded groups can
be neoclassical; and that all but finitely many of these non-neoclassical noded
Schottky groups are sufficiently complicated; that is, they are so far from
being neoclassical that ``nearby'' Schottky groups are also not classical.
It is an open question whether every closed Riemann surface can be uniformized
by a classical Schottky group. We give a family of examples in genus $3$ of
noded Riemann surfaces that cannot be uniformized by neoclassical Schottky
groups, thus answering in the negative the corresponding question for noded
Riemann surfaces in the Deligne-Mumford compactification of moduli space.
The basic ideas used in the proofs of the above facts stem from the following
facts. Let $G$ be a noded Schottky group (i.e., a free, geometrically finite
Kleinian group), with infinitely many components (of its set of discontinuity),
all simply-connected. Then, if $G$ is free on $p$ generators, there are $2p$
almost disjoint simple closed curves on the Riemann sphere bounding a
disconnected region, so that there are generators of $G$ pairing these curves.
Such curves are called {\em defining loops} for $G$.
Since each component is simply connected, every defining loop must pass through
some number of components; it must enter and leave each such component by
passing through a parabolic fixed point. It is essentially immediate that if two
distinct translates of some defining loop or loops both pass through two
distinct parabolic fixed points, then they cannot both be circles. In
particular, if two translates of some defining loop or loops both pass through
some component, entering and leaving the component by the same parabolic fixed
points, then they cannot both be circles. We will show below that if there are
two distinct translates of some defining loop or loops which pass through two
successive components, entering and leaving each through the same parabolic
fixed points, then there are necessarily four points on at least one of them so
that the imaginary part of the cross-ratio of these four points is uniformly
bounded away from zero.

One can regard the noded Schottky groups as being on the boundary of
the space of Schottky groups. Other boundaries of the algebraic topology
space of algebraically marked Schottky groups have been studied by, among
others, Gerritzen and Herrlich \cite{Gerritzen:Schottky}, Hidalgo
\cite{Hidalgo:NodedSchottky} and Sato \cite{Sato:Augmented}.

%%%%%%%%%%%%%%%%%%%%%%%%%%%
%%%%%%%%%%%%%%%%%%%%%%%%%%%
\section{Schottky groups and noded Schottky groups}
%%%%%%%%%%%%%%%%%%%%%%%%%%%%
\subsection{Kleinian groups} Throughout this paper, we will use the term Kleinian group to denote a discrete subgroup of $\PSLC$, which we identify as both the group of orientation-preserving conformal homeomorphisms of the extended complex plane, $\Ch$, and as the group of orientation-preserving isometries of hyperbolic 3-space. 

Every Kleinian group $G$ decomposes $\Ch$ into two sets; the limit set $\Lambda=\Lambda(G)$, and its complement, the regular set (or discontinuity set)  $\Omega=\Omega(G)$. The limit set is defined as the set of limit points of the set $\{g(x):g\in G\}$, where $x$ is any point in hyperbolic 3-space.

It is well known that the limit set is either finite, consisting of at most two points, or is an uncountable perfect nowhere dense subset of $\Ch$, or is equal to $\Ch$. The Kleinian group is {\em elementary} if $\Lambda$ consists of at most two points. It is of the {\em first kind} if $\Lambda=\Ch$; it is of the {\em second kind} otherwise.

%%%%%%%%%%%%%%%%%%%%%%%
\subsection{Definitions of Schottky groups}
There are several equivalent definitions of a Schottky group; we start with
the traditional one.

\s
\noindent
{\bf Definition 1:} A Schottky group is a Kleinian group $G$
generated by transformations $a_1,\ldots,a_p$, where there are $2p$ disjoint
simple loops, $C_1,C'_1,\ldots,C_p,C'_p$, bounding a common domain $D$ in the
extended complex plane $\Ch$, where $a_i(C_i)=C'_i$, and $a_i(D)\cap
D=\emptyset$, $i=1,\ldots,p$.

\s

In the above, the set of loops, $C_1,\ldots,C'_p$ are called
{\em defining loops}
for the {\em Schottky generators} $a_1,\ldots,a_p$. If we do not need the
generators, we will sometimes simply refer to the set of loops as defining
loops. The value $p$ is called the genus of the Schottky group and we
sometimes refer to $G$ as a Schottky group of genus $p$.
Let $\Omega=\Omega(G)$ be the set of discontinuity of $G$.
It is well known that $\Omega=\Omega(G)$ is connected and dense in $\Ch$,
and that $S=\Omega/G$ is a closed Riemann surface of genus $p$.
Also, if we let $V_i$ be the projection of $C_i$ to
$S$, then $V_1,\ldots,V_p$ is a set of simple disjoint homologically
independent loops. We also call $V_1,\ldots,V_p$ a {\em set of defining
loops} for $G$.
It is also well known that a Schottky group is finitely generated, free and purely
loxodromic. It was shown in \cite{bm:schottky} that these properties
characterize these groups within the class of Kleinian groups of the second kind.

\s
\noindent
{\bf Definition 2:} A Schottky group is a finitely generated, free Kleinian
group of the second kind, in which every non-trivial element is loxodromic
(including hyperbolic).
It is well known that every Schottky group is geometrically finite. A
geometrically finite Kleinian group without parabolics is either co-compact or
of the second kind. Since a free group cannot be the fundamental group of a
closed 3-manifold, we can restate this second definition as follows.

\s
\noindent
{\bf Definition 2':} A Schottky group is a free geometrically finite Kleinian
group without parabolics. 

\s

The following equivalent definition is essentially Koebe's retrosection theorem;
a modern proof can be found in \cite[pp. 312 ]{bm:book}.

\s
\noindent
{\bf Definition 3:} Let $S$ be a closed Riemann surface of genus $p$, and
let $V_1,\ldots,V_p$ be $p$ homologically independent disjoint simple closed
loops on $S$. Each of these loops is homologically non-trivial, hence also homotopically non-trivial.
Then the highest regular covering of $S$ for which
$V_1,\ldots,V_p$ all lift to loops is conformally an open subset of $\Ch$; the
group of deck transformations for this covering, as a group of conformal maps
of a plane domain, is a Schottky group of genus $p$.

\s

We include the fourth definition, which we will not use, for the sake of
completeness; its equivalence with definition 2 is well known.

\s
\noindent
{\bf Definition 4:} Let $M$ be a handlebody of genus $p$ endowed with a
geometrically finite complete hyperbolic structure, where the injectivity radius
is bounded from below by a positive constant. Then the universal covering group of $M$, acting on
hyperbolic $3$-space $\Hh$, is a Schottky group of genus $p$.

%%%%%%%%%%%%%%%%%%%%%%%%%%%
\subsection{Definitions of noded Schottky groups} 
It is clear that if we permit
the defining loops of a Schottky group to touch without crossing at a finite
number of points, we will still generate a free Kleinian group of the second
kind. However, in order to guarantee that the common outside of all the curves
is a fundamental domain, we need to ensure that these new points of contact
are parabolic fixed points (see \cite{bm:combiv}).

\s
\noindent
{\bf Definition 1:} A {\em noded Schottky group} is a Kleinian group $\widehat
G$ generated by M\"obius transformations $\widehat a_1,\ldots,\widehat a_p$,
where there are $2p$ simple loops, $\widehat C_1,\widehat C'_1,\ldots,\widehat
C_p,\widehat C'_p$, satisfying the following conditions:
\begin{itemize}
\item[(i)] Any two of these loops intersect in at most finitely many points;
\item[(ii)] there are pairwise disjoint open discs, $\widehat D_{1}$,
$\widehat D'_{1}$,..., $\widehat D_{p}$, $\widehat D'_{p}$, so that the boundary
of $\widehat D_{i}$ is $\widehat C_{i}$
and the boundary of $\widehat D'_{i}$ is $\widehat C'_{i}$;
\item[(iii)] $\widehat a_{i}(\widehat C_{i})=\widehat C'_{i}$, and
$\widehat a_i(\widehat D_i)\cap \widehat D'_i=\emptyset$, $i=1,...,p$;
\item[(iv)] each point of intersection of any two defining loops
is a fixed point of a parabolic transformation in $\widehat G$.
\end{itemize}

\s

As above, the loops $\widehat C_i$, $\widehat C'_i$ are called {\em defining loops}
for the noded Schottky group, and the generators, $\widehat a_1,\ldots,\widehat a_p$
are called {\em noded Schottky generators}. The value $p$ is called the genus
of $G$ and we also say that $G$ is a noded Schottky group of genus $p$.
It was shown in \cite{bm:combiv} that every noded Schottky group is geometrically
finite; that every parabolic element of $\widehat G$ is a
conjugate of a power of one of the transformations fixing a common point of
two of the defining loops and, if we let $\widehat D$ denote the complement of
the union of the closures of $\widehat D_{1}$, $\widehat D'_{1}$,..., $\widehat D_{p}$,
$\widehat D'_{p}$, that $\widehat D$ is a fundamental domain for $\widehat G$.
The fact that a geometrically finite free Kleinian group of the second kind is
a noded Schottky group was proved in \cite{bm:free}. The fact that every
geometrically finite free Kleinian group is of the second kind was proved in
\cite{Hidalgo:NodedSchottky}.

\s
\noindent
{\bf Definition 2:} A {\em noded Schottky group} is a geometrically finite
free Kleinian group.

\s

Since we will not need it, we will not state the (complicated and difficult)
analog of definition 3 for noded Schottky groups. However, we do note that the
analog of definition 4 is quite simple.

\s
\noindent
{\bf Definition 4}: Let $M$ be a handlebody of genus $p$ with a geometrically
finite hyperbolic structure. Then the universal covering group of $M$, acting
on $\Hh$, is a noded Schottky group of genus $p$.

\s

In the definitions above, every Schottky group is also a noded Schottky group.
A {\em proper} noded Schottky group is one that necessarily contains
parabolic elements.

%%%%%%%%%%%%%%%%%%
\subsection{Generators vs. groups}
Throughout this paper, we will deal with
Schottky groups and noded Schottky groups, sometimes marked topologically by
a set of defining loops, including a set of generators, sometimes only marked
by a set of generators. In any case, we will usually regard these generators as
being given up to conjugation in $\PSLC$.
Chuckrow's theorem \cite{chuckrow:schottky} asserts that every set of free
generators for a Schottky group is a set of Schottky generators; that is, if
$a_1,\ldots,a_p$ is any set of free generators for a Schottky group, then
there are defining loops $C_1,\ldots,C'_p$ for these generators as in
definition 1. One expects however that the corresponding statement is not true
for noded Schottky groups (this will be explained below).
Let $\widehat G$ be a noded Schottky group of genus $p$. If we adjoin the
parabolic fixed points of $\widehat G$ to $\Omega(\widehat G)$, with the appropriate
cusped topology (see \cite{Hidalgo:NodedFuchsian} or \cite{K-M:pinched}), we
obtain the set $\Omega^+(\widehat G)$, called the noded set of discontinuity of
$\widehat G$; one then easily sees that $S^+=\Omega^+/\widehat G$ is a noded closed
Riemann surface of genus $p$. It was observed in \cite{Hidalgo:NodedSchottky}
that, given any noded closed Riemann surface $S^{+}$ of genus $p$, there is a
noded Schottky group $\widehat G$ (necessarily of genus $p$) so that
$S^+=\Omega^+/\widehat G$. It follows that every point of the Deligne-Mumford
compactification of the moduli space of genus $p$ can be realized by a noded
Schottky group of genus $p$.
It was
shown in \cite{bm:free} that there are $p$ (smooth) loops on $S^+$, $\widetilde
W_1,\ldots,\widetilde W_p$, which lift to loops on $\Omega^+$, and
which, except that they may pass through the nodes, are simple and disjoint;
at the nodes, they touch but do not cross (that is, if two of these loops pass through a node, then they are tangent to each other at each side of the node; since the two sides of the node are consistently oriented, and the curves are smooth, one can follow them across the node to see if they cross or not).
By using Chuckrow's theorem, one easily sees that the result in \cite{bm:free}
in fact yields the following. Let $\widehat a_1,\ldots,\widehat a_p$ be any set of
free generators of $\widehat G$. Then we can choose the $p$ loops, $\widetilde
W_1,\ldots,\widetilde W_p$, on $S^+$, so that if we lift the complement of
these loops (including the appropriate nodes) to $\Omega^+(\widehat G)$, then
we obtain a fundamental domain for $\widehat G$, bounded by $2p$ not necessarily
simple loops. These $2p$ loops come in pairs, where each pair, say
$\widehat C_i$ and $\widehat C'_i$, are both lifts of the same
$\widetilde W_i$; further, for each $i=1,\ldots,p$, this pairing is explicitly
accomplished by the fact that $\widehat a_i(\widehat C_i)=\widehat C'_i$. As
suggested above, while these loops, $\widehat C_1,\widehat
C'_1,\ldots,\widehat C_p,\widehat C'_p$, are almost disjoint (that is, they
meet without crossing at some number of parabolic fixed points), they need not
be simple --- that is, some of them might pass through the same parabolic
fixed point more than once. It might even be that if we start with a given set
of free generators for a noded Schottky group, then there is no choice of a
set of simple almost disjoint defining loops for these generators.

%%%%%%%%%%%%%
\subsection{Rigid groups} 
In general, a {\em rigid} Kleinian group is one
that admits no (quasiconformal) deformations. It is well known that a
geometrically finite Kleinian group of the second kind $G$ is rigid if
and only if $\Omega(G)/G$
is a finite union of thrice punctured spheres.
We will need the fact that if $\widehat G$ is a rigid noded Schottky group,
then every component of $\widehat G$ (i.e., connected component of
$\Omega(\widehat G)$) is a Euclidean disc, and, if one factors this Euclidean
disc by its stabilizer, one obtains a thrice punctured sphere. The Fuchsian
group stabilizing this Euclidean disc is a conjugate of the classical
Fuchsian $(\infty,\infty,\infty)$-triangle group.

%%%%%%%%%%%%%%%%%%%%%%%%
\subsection{Classical and neoclassical groups} 
If $G$ is a Schottky group, given
by a set of defining loops all of which are Euclidean circles, then $G$ is
{\em classical}, and the corresponding generators form a {\em classical}
generating set. If $\widehat G$ is a noded Schottky group given by a set of defining loops all of which are Euclidean circles, then $\widehat G$ is a {\em neoclassical} Schottky group, and the corresponding set of generators is a {\em neoclassical} set of generators. 

There are Schottky groups for which every set of free generators is classical. For example, every finitely generated purely hyperbolic Fuchsian group representing a closed surface with holes, is, qua Kleinian group, a Schottky group. In the special case that the Fuchsian group is a two generator group representing a torus with one hole, this Schottky group is classical on every set of generators.

On the other hand, regardless of whether $G$ is classical or not, if $a_1,\ldots,a_p$ is a set of free generators for $G$, and no set of defining loops with these generators consists only of Euclidean circles, then $a_1,\ldots,a_p$ forms a {\em non-classical} set of generators. For example, if the Fuchsian group represents a sphere with three holes, then it has infinitely many sets of generators that are classical and infinitely many sets of generators that are non-classical. One expects that in general, classical Schottky groups have many sets of generators that are non-classical.  

For noded Schottky groups, if $\widehat a_1,\ldots, \widehat a_p$ is a set of free generators for the noded Schottky group $\widehat G$, and every set of defining loops for these generators contains at least one loop that is not a Euclidean circle, then $\widehat a_1,\ldots, \widehat a_p$ is a {\em non-neoclassical} set of generators. As with Schottky groups, one expects that in general a noded Schottky group will have many sets of generators that are not neoclassical.
One also expects that in general a neoclassical Schottky group will have many sets of generators that are not neoclassical.

The Schottky group $G$ is {\em classical} if it has at least one set of classical generators; it is non-classical otherwise. Similarly, the noded Schottky group $\widehat G$ is {\em neoclassical} if it has at least one set of neoclassical generators; it is non-neoclassical otherwise.

%%%%%%%%%%%%%%%%%%%%%%
%%%%%%%%%%%%%%%%%%%%%%
\section{The boundary of Schottky space} 
There are several natural spaces of
Schottky groups (see \cite{bm:spaces}), and one can regard some of our
constructions here as that of putting the appropriately marked noded Schottky
groups on the boundaries of these spaces. These ideas will be pursued
elsewhere.
We will need the usual space of deformations of a Schottky group of genus $p$,
denoted by $\Salg$, which we will usually think of as a subset of the
representation space of the free group of rank $p$ in $\PSLC$, modulo
conjugation.

%%%%%%%%%%%%%%%%%%%%%%%
\subsection{The infinite shoebox construction}\label{sec:shoebox}
Let $\widehat C_1,\ldots,\widehat C'_p$, with generators $\widehat a_1,\ldots,\widehat a_p$,
be a set of defining loops for a noded Schottky group $\widehat G$, and let $\widehat
p_1,\ldots,\widehat p_q$ be a maximal set of primitive parabolic elements of $\widehat
G$ generating non-conjugate cyclic subgroups. For each $i=1,\ldots,q$, choose
a particular M\"obius transformation $h_i$ conjugating $\widehat p_i$ to the
transformation $P(z)=z+1$. Consider the renormalized group
$h_i\widehat Gh_i\ngn$. For this group, there is a number $\alpha _0>1$ so
that the set $\{|\Im(z)|\ge\alpha _0\}$ is precisely invariant under
$\Stab(\infty)\subset h_i\widehat Gh_i\ngn$.
Regard $\Hh$ as being the set $\{(z,t):z\in{\mathbb C},
t>0\in{\mathbb R}\}$. We likewise identify ${\bbb C}$ with the boundary of $\Hh$, except for the point at infinity; that is, we identify ${\bbb C}$ with $\{(z,t):t=0\}$.

In this normalization, for each parameter $\alpha$,
with $\alpha >\alpha _0$, we define the {\em infinite shoebox} to be the set
$B_{0,\alpha}=\{(z,t):|\Im(z)|\le \alpha ,t\le\alpha \}$. Since $\alpha_0>1$, we easily observe that for every $\alpha>\alpha _0$, the complement of $B_{0,\alpha}$ in $\Hh\cup{\mathbb C}$ is precisely invariant under $\Stab(\infty)\subset h_i\widehat Gh_i\ngn$, where we are now
regarding M\"obius transformations as hyperbolic isometries, which act on the closure of $\Hh$. Then for $\widehat G$, the infinite shoebox with parameter $\alpha$ at $z_i$,
the fixed point of $\widehat p_i$, is $B_{i,\alpha}=h_i\ngn(B_{0,\alpha})$,
and, if $\widehat p$ is any parabolic element of $\widehat G$, conjugate to
some power of $\widehat p_i$, then the corresponding infinite shoebox at $x$,
the fixed point of $\widehat p$, is given by $b(B_{i,\alpha })$, where
$\widehat p=b\widehat p_ib\ngn$.

It was observed in \cite{bm:free} that, for each fixed $\alpha>\alpha _0$,
$\widehat G$ acts as a group of conformal homeomorphisms on the {\em expanded
regular set} $B^{\alpha}=\bigcap \widehat g(B_{i,\alpha})$, where the intersection is
taken over all $\widehat g\in \widehat G$ and all $i=1,\ldots,q$. Further, $\widehat G$
acts as a Schottky group (in the sense of definition 3) on the boundary of
$B^{\alpha }$. Observe that each parabolic $\widehat g\in\widehat G$ appears
to have two fixed points on the boundary of $B^{\alpha}$; that is, $\widehat
g$, as it acts on the boundary of $B^{\alpha}$, appears to be loxodromic.
We define the {\em flat part} of $B^{\alpha }$ to be the intersection of
$B^{\alpha}$ with the extended complex plane. The complement
of the flat part (on the boundary of $B^{\alpha}$) is the disjoint union of
3-sided {\em boxes}, where each box has two {\em vertical sides}
(translates of the sets $\{\Im(z)=\pm \alpha , 0<t<\alpha \}$) and one {\em
horoball side} (a translate of the set $\{|\Im(z)| \leq \alpha , t=\alpha \}$).

We will need a particular class of relatively compact subsets of $B^{\alpha }$.
For each $i=1,...,q$ and for each $n=1,2,...$,
we define $B_{i,\alpha ,n}$ to be $h_{i}^{-1}(B_{0,\alpha } \cap
\{|\Re(z)|\leq n\})$. Then we define
$B^{\alpha,n}=\bigcap \widehat{g}(B_{i,\alpha,n})$, and the {\em truncated flat part} of
$B^{\alpha,n}$ as the intersection of $B^{\alpha ,n}$ with the extended
complex plane. The boundary of the truncated flat part near a parabolic fixed point,
renormalized so as to lie at $\infty$, is a Euclidean rectangle.

We now renormalize $\widehat G$ so that $\infty$ is an interior point of
$\Omega(\widehat G)$.
Then, for each $\alpha>\alpha _0$, there is a conformal map $f^{\alpha}$,
mapping the boundary of $B^{\alpha}$ to $\Ch$, and conjugating $\widehat G$ onto a
Schottky group $G^{\alpha }$, where $f^{\alpha }$ is defined by the
requirement that, near $\infty$, $f^{\alpha}(z)=z+O(|z|\ngn)$. We remark that the group $G^{\alpha}$ depends on the choice of the transformations $h_1,\ldots,h_q$ as well as on the choice of $\alpha $.

It was shown in \cite{bm:free} that, with the above normalization, $f^{\alpha}\to I$, where $I$
denotes the identity, uniformly on compact subsets of $\Omega(\widehat G)$, and
that, for each fixed $\widehat g\in\widehat G$, $f^{\alpha}\widehat
g(f^{\alpha})\ngn\to \widehat g$, as $\alpha \to \infty$. In particular, if we
fix $\alpha _0$, and fix $n$, then $f^{\alpha }\to I$ uniformly on compact subsets of the
truncated flat part of $B^{\alpha_0,n}$.

The boundary of $B^{\alpha _0,n}$ consists of a disjoint union of quadrilaterals
with circular sides. After renormalization, the part of the boundary of
$B^{\alpha _0,n}$ corresponding to $\{|\Im(z)|=\alpha _0\}$ is the {\em
horizontal} part of the boundary, while the part of the boundary corresponding
to $\{|\Re(z)|=n\}$ is the {\em vertical} part of the boundary.

%%%%%%%%%%%%%%%%%%%
\subsection{Vertical projection}\label{sec:vertical} 
We fix a choice
of conjugating maps, $h_i$, $i=1,\ldots,q$, and we fix a choice of the parameter 
$\alpha>\alpha_0$. 

We need to deform all the $\widehat C_i$ and $\widehat C'_i$, within 
$\Omega^+(\widehat G)$ to an equivalent defining set of loops, with the same 
generators, so that, after appropriate renormalization, each connected 
component of each of the deformed loops appears, in each component of the 
complement of the flat part of $B^{\alpha}$, as a pair of half-infinite 
Euclidean vertical lines, one in $\{\Im(z)\ge \alpha\} $, the other in $\{\Im(z)\le 
-\alpha\} $, both with the same real part. This deformation is accomplished in several steps. We start by deforming the $\widehat C_i$ and $\widehat C'_i$ so that they all meet the lines $\{\Im z=\alpha_0\}$ and $\{\Im z=-\alpha_0\}$ transversely in at most finitely many points, and that these points of intersection are all distinct.

We first observe that we can easily deform all the $\widehat C_i$ and $\widehat C'_i$, so that each of them intersects both $\{\Im (z)>\alpha_0\}$, and $\{\Im (z)<-\alpha_0\}$, in at most one arc having one of its endpoints at the point at infinity. To accomplish this, we remark that we have required that each $C_i$ and each $C'_i$ be simple, and so can pass at most once through $\infty$. Then, if there is another arc of some $C_i$ intersecting for example $\{\Im (z)>\alpha\}$, this arc starts and ends at two points on $\{\Im (z)=\alpha_0\}$, and the distance between these points is less than $1$. There are at most a finite number of other such arcs of the $C_i$ or $C'_i$ lying between these two points of intersection, and there is one whose points of intersection on this line are closest. Since there is some $\epsilon>0$ so that the set $\{\Im (z)>\alpha_0-\epsilon\}$ is precisely invariant under $\Stab(\infty)$, we can easily pull this arc, and its appropriate image, down below the line $\{\Im (z)=\alpha_0\}$ while maintaining all the requisite properties for these loops. After a finite number of iterations of this step, we will have accomplished that each $\widehat C_i$, each each $\widehat C'_i$, intersects transversally each of the lines $\{\Im z=\alpha_0\}$ and $\{\Im z=-\alpha_0\}$ at most once.

We now have some number of points of intersection of the loops $C_i$ and $C'_i$ with the line $\{\Im (z)=\alpha_0\}$; call these $x_1,\ldots,x_m$, where $\Re(x_1)<\cdots<\Re(x_m)$. There are likewise $m$ points of intersection of these loops with the line $\{\Im (z)=-\alpha_0\}$; call these points of intersection $y_1,\ldots,y_m$, where $\Re(y_i)<\cdots<\Re(y_m)$.

Our next observation is that $\Re(x_m)\le\Re(x_1)+1$. The loops $\widehat C_k$ and $\widehat C'_k$ all lie on the boundary of the fundamental domain $\widehat D$. So the arcs of the line $\{\Im (z)=\alpha_0\}$ lying between these points of intersection are alternately in $\widehat D$ and in its complement. Since the transformation $P:z\to z+1$ must map $\widehat D$ into its complement, the region to the left of $x_1$, and the region to the right of $x_m$ both lie in the complement of $\widehat D$. Now assume that $\Re(P(x_1))<\Re(x_m)$. Since $P(\widehat D)$ must be contained in the complement of $\widehat D$, we must have that, for some $k<m$,  $P(x_1)=x_k$. Since the region to the right of $x_1$ lies in $\widehat D$, and the region to the left of $x_m$ lies in $\widehat D$, we cannot have that $P(x_1)=x_{m-1}$.  Hence $k<m-1$, from which it follows that $P(x_2)=x_{k+1}$. Since each $\widehat C_j$ and each $\widehat C'_j$ has at most one point of intersection with the line $\{\Im (z)=\alpha_0\}$, we now have that two of the identifications of the $\widehat C_j$ with the $\widehat C'_j$ are accomplished with either $P$ or $P\ngn$, contradicting the fact that these identifications are free generators of the renormalized $\widehat G$.

The above argument also shows that either $P(y_1)=y_m$, or $\Re(P(y_1))>\Re(y_m)$. Of course, we must have that $P(y_1)=y_m$ if and only if $P(x_1)=x_m$.

In the case that $P(x_1)=x_m$, we define, for $k=1,...,m$, the points
$$x'_k=\frac{k-1}{m-1}+i\alpha, \quad \mbox{and} \quad y'_k=\frac{k-1}{m-1}-i\alpha.$$

In the case that $\Re(P(x_1))>\Re(x_m)$, we define, for $k=1,\ldots,m$, the points
$$x'_k=\frac{k-1}{m}+i\alpha, \quad \mbox{and} \quad y'_k=\frac{k-1}{m}-i\alpha.$$

We next observe that if $\widehat C_j$ intersects the line $\{\Im (z)=\alpha_0\}$ at $x_k$, then it also intersects the line $\{\Im (z)=-\alpha_0\}$ at $y_k$, for no pair of the loops, $\widehat C_j$, $\widehat{C}_{i}$, $\widehat{C}'_i$, $\widehat C'_j$, or any of their translates, cross at $\infty$. 

We are now in a position to choose the deformed $\widehat C_j$ and $\widehat C'_j$ as follows. If $\widehat C_j$ intersects the line $\{\Im (z)=\alpha_0\}$ at say $x_k$, then we deform it so that it is linear from $x_k$ to $x'_k$, then linear, with the same real part, to $\infty$, then linear with the same real part to $y'_k$, then linear to $y_k$. 

The {\em vertical projection} of these deformed loops
consists of replacing each of the above pairs of half-infinite lines (i.e., the lines going from some $x'_k$ to $\infty$, or from some $y'_k$ to $\infty$) with a trio
of line segments, all with the same real part. Two of these line segments keep
 the line $\{\Re(z)=x\}$ fixed, and run up the vertical sides of the shoebox, one at the
point $x+i\alpha $, the other at the point $x-i\alpha $; up to direction, the
lines are then defined as $\{z=x+i\alpha $, $0\le t\le\alpha \}$, and $\{z=x-i\alpha
$, $0\le t\le\alpha\} $. The third line segment runs along the horizontal part of
the boundary and can be described as $\{\Re(z)=x$, $-\alpha \le \Im(z)
\le\alpha $, $t=\alpha\} $.

Observe that the vertical projection of the defining loops
$\widehat C_1,\ldots,\widehat C'_p$, for $\widehat G$ yields a set of defining loops,
$C_1^{\alpha},\ldots,(C'_p)^{\alpha,}$ with corresponding generators
$a_1=f^{\alpha}\widehat a_1(f^{\alpha})\ngn,\ldots,a_p=f^{\alpha}\widehat
a_p(f^{\alpha})\ngn$, for the Schottky group $G^{\alpha }$.
For each fixed choice of the transformations $h_1,\ldots,h_q$,
and each fixed choice of
$\alpha _0$, where $\alpha _0$ is chosen as above, we define the {\em relative conical
neighborhood} of $\widehat G$ to be the set of all marked Schottky groups
$G^{\alpha}=\langle a_1,\ldots,a_p\rangle$.

We will assume the above normalizations from now on. That is, we normalize
$\widehat G$ so that $\infty$ is an interior point of the flat part
corresponding to $\alpha_0$, and we normalize each $f^\alpha $ so that
$f^\alpha (z)=z+O(|z|\ngn)$ near $\infty$. With these normalizations,
$f^\alpha \to I$ uniformly on compact subsets of $\Omega(\widehat G)$, and
$G^\alpha \to \widehat G$ algebraically. We remark that it follows from the
J{\o}rgensen-Marden criterion \cite{J-M:geoconv} that $G^\alpha \to G$
geometrically.
We also remark that it is unclear how large the relative conical neighborhoods are.
However, since each $G^\alpha $ is a Schottky group, and, for $\alpha \to\infty$,
$G^\alpha \to G$, it follows that each relative conical neighborhood contains
infinitely many distinct marked Schottky groups. It is also easy to see, as in
\cite{bm:free}, that, for each primitive parabolic element $\widehat
g\in\widehat G$, as $\alpha \to\infty$, there is a corresponding geodesic on
$S^\alpha =\Omega(G^\alpha )/G^\alpha $ whose length tends to zero. It follows
that each relative conical neighborhood of a noded Schottky group contains Schottky
groups representing infinitely many distinct Riemann surfaces.

%%%%%%%%%%%%%%%%%%%%%%%%%%
%%%%%%%%%%%%%%%%%%%%%%%%%%
\section{Sufficiently Complicated Noded Schottky Groups} 
%%%%%%%%%%%%%%%%%%%%%%
\subsection{Pinchable sets of geodesics}
Let $\widehat G$ be a noded Schottky group
(of genus $p\ge 2$) defined by the set of defining loops, $\widehat C_1,\ldots,\widehat
C'_p$, with corresponding generators $\widehat a_1,\ldots,\widehat a_p$. This set of
generators defines a representation of the free group on $p$ elements,
${\mathcal F}_p$, into $\PSLC$; the set of elements of the free group mapped
onto parabolic elements of $\widehat G$ are the {\em pinched} elements.
Let $G$ be a Schottky group with defining loops $C_{1},\ldots,C'_{p}$,
and generators $a_{1},\ldots,a_{p}$, where these generators for $G$ lie
in one of the relative conical neighborhoods of the generators $\widehat a_1,\ldots,\widehat a_p$ for $\widehat G$ given in
Sect. \ref{sec:vertical}, and where these defining loops have been obtained from
$\widehat C_1,\ldots,\widehat C'_p$ using vertical projection (see \ref{sec:vertical}).
Let $S=\Omega(G)/G$ be the closed Riemann surface of genus $p$ represented by $G$, and
let $V_i$ be the projection of $C_{i}$, $i=1,\ldots,p$. Then $V_1,\ldots,V_p$
is a set of $p$ homologically independent simple disjoint loops on $S$.
Let $\psi:G\to\widehat G$ be the isomorphism defined by $a_{i}\mapsto \widehat a_i$,
$i=1,\ldots,p$. It is essentially immediate that there are simple disjoint
geodesics $L_1,\ldots,L_q$ on $S$, defined by the words $W_1,\ldots,W_q$ in
the generators $a_{1},\ldots,a_{p}$, so that $\psi(W_1),\ldots,\psi(W_q)$ are
all parabolic in $\widehat G$, and every parabolic element of $\widehat G$ is a power
of a conjugate of one of these.

We remark that the construction in \cite{bm:free} shows that, given the noded Schottky group $\widehat G$, with its set of generators, we can choose the parameter $\alpha$, which determines the Schottky group $G$, with its corresponding set of generators, and we can choose the $L_i$ on $\Omega(G)/G$, so that the $L_i$ are all arbitrarily short.
It was further shown in \cite{bm:parelt}
(see also Yamamoto \cite{yamamoto:parelt})
that if $G$ is any Schottky group, with generators $a_{1},\ldots,a_{p}$, and
$L_1,\ldots,L_q$ is any set of simple disjoint geodesics on $S$, defined by
the words $W_1,\ldots,W_q$, as above, where no $W_i$, as an element of $G$, is a non-trivial power, and, for $i\ne j$, the cyclic subgroups generated by $W_i$ and $W_j$ are not
conjugate in $G$, then there is a noded Schottky group $\widehat G$, and there is
an isomorphism $\psi:G\to \widehat G$, where $\psi(W_1),\ldots,\psi(W_q)$, and
their powers and conjugates, are exactly the parabolic elements of $\widehat G$.
More precisely, it was shown in \cite{bm:parelt} that there is a path in
Schottky space, $\Salg$, which converges to a set of generators for $\widehat G$,
along which the lengths of the geodesics, $L_1,\ldots,L_q$, all tend to zero.

We now fix a Schottky covering $\Omega(G)\to S$ of a closed Riemann surface $S$,
and we fix a set of generators, $a_{1},\ldots,a_{p}$, for $G$.
If $L_1,\ldots,L_q$ is any set of simple disjoint geodesics on $S$
where no $L_i$, as an element of $G$, is a non-trivial power,
and $L_i$ and $L_j$ generate
non-conjugate cyclic subgroups of $G$, for all $i\ne j$, then we say that this
set of geodesics is {\em pinchable}.

\s
\noindent
\begin{Prop}\label{prop:addpinch} Given a non-empty set of $k<3p-3$ pinchable
geodesics, $L_1,\ldots,L_k$ on $S$, there is a set of $k+1$ pinchable
geodesics, $L_1,\ldots,L_k,L_{k+1}$.
\end{Prop}
\begin{proof} Since $L_1,\ldots,L_k$ is pinchable, we can find a noded Schottky
group $\widehat G$ for which these are pinched. Since $k<3p-3$, one of the
components of $\widehat G$, call it $\Delta$, is such that $\widehat
S=\Delta/\Stab(\Delta)$ is not a thrice punctured sphere.
Let $H=\Stab(\Delta)\subset\widehat G$. By Ahlfors' finiteness theorem, since $\Delta$ is a
component subgroup of the finitely generated group $\widehat G$, $H$ is finitely
generated. Since $H$ is a subgroup of the free group $\widehat G$, it is also free.
We first take up the case that $\Delta$ is simply connected, in which case 
$\widehat H$ 
is a quasifuchsian group of the first kind. In this case, since $H$ is free, 
we can find generators for $H$, $A_1,B_1,\ldots,A_q,B_q,P_1,\ldots,P_r$, where 
these generators all represent simple loops on $\widehat S$, and where the $A_i$ 
and $B_i$ are loxodromic (including hyperbolic); the $P_i$ are parabolic; 
these generators satisfy the one defining relation: 
$[A_1,B_1]\cdots [A_q,B_q]\cdot P_1\cdots P_r=1$; and every parabolic element 
of $H$ is conjugate to a power of some $P_i$. If $q>0$, then $A_1$ is 
pinchable; if $q=0$, then $r\ge 4$ and $P_1P_2$ is loxodromic and pinchable.
If $\Delta$ is not simply connected, then we can write $H$ as a free product, 
in the sense of combination theorems, $H=H_1\ast\cdots\ast H_q\ast F_r$, 
where each of $H_1,\ldots,H_q$ is a quasifuchsian group, and $F_r$ is a 
Schottky or Schottky-type group. 
We first take up the case that $q=0$; this is the case that $\Delta=\Omega(H)$. In 
this case, $G$ is the free product, in the sense of combination theorems, of $p$ 
infinite cyclic groups. If one of these cyclic groups is loxodromic, then we can 
pinch the corresponding generator. If all $p$ generators are parabolic, call 
them $P_1,\ldots,P_p$, then we can pinch either $P_1P_2$ or $P_1P_2\ngn$.
We next take up the case that $q>0$. 
If some $H_i$ does not represent a thrice punctured sphere, 
then we proceed as above. 
If every $H_i$ represents a thrice punctured sphere and $q>1$, 
then we find geometric generators 
(i.e., parabolic generators as above) 
$P_{1,1}$, $P_{1,2}$ and $P_{1,3}$ for $H_1$, and parabolic generators 
$P_{2,1}$, $P_{2,2}$ and $P_{2,3}$ for $H_2$ as above; that is, 
$P_{i,1}P_{i,2}P_{i,3}=1$, $i=1,2$. 
Then the free group $H$ has rank at least $4$, 
and either $P_{1,1}P_{2,1}$ or $P_{1,1}P_{2,1}\ngn$ is a primitive loxodromic 
element of $H$ which represents a simple loop; hence it is pinchable. 
If $q=1$, and $H_1$ represents a thrice punctured sphere, 
then since $H_1\neq H$, $r\ge 1$. 
If $F_r$ has at least one loxodromic generator, then that generator is 
pinchable. Otherwise, we write the generators of $H_1$ as $P_{1,1}$, $P_{1,2}$ 
and $P_{1,3}$ as above, and let $P$ be one of the parabolic generators of 
$F_r$. Then either $P_{1,1}P$ or $P_{1,1}P\ngn$ represents a simple loop on 
$\widehat S$, and so is pinchable. \end{proof}

%%%%%%%%%%%%%%%%%
\subsection{Valid sets of defining loops} 
Let the Schottky group $G$, uniformizing the Riemann surface $S$ be as above. Let 
$L_1,\ldots,L_q$ be a pinchable set of geodesics on $S$, and let $\widehat S^+$ be the 
noded Riemann surface obtained from $S$ by pinching these $q$ geodesics; 
$\widehat S^+$ consists of a finite number of compact Riemann surfaces, 
called {\em parts}, which are joined together at a 
finite number of nodes. Also, let $\widehat G$ be the noded Schottky group obtained 
from $G$ by pinching these $q$ geodesics; 
let $\widehat\Omega$ be the set of discontinuity of 
$\widehat G$, and let $\widehat\Omega^+$ be the noded set of discontinuity of $\widehat G$. 
Let $V_1,\ldots,V_p$, be a set of defining loops on $S$ and let $\widehat 
V_1,\ldots,\widehat V_p$ be the loops on $\widehat S^+$ obtained by pinching 
$L_1,\ldots,L_q$. We observe that the lifts of the $\widehat V_i$ to 
$\widehat\Omega^+$ are all loops, but they are generally not disjoint, and they need not be simple. There are certainly some number of these lifts passing through each parabolic fixed point, and  some of them 
might pass more than once through the same parabolic fixed point. The set of 
loops, $V_1,\ldots,V_p$, is a {\em valid set of defining loops} for 
$L_1,\ldots,L_q$, if every lift of every $\widehat V_i$ to $\widehat\Omega^+$ is a 
simple loop; that is, it passes at most once through each parabolic fixed point. In this case, we note that the set of loops, $\widehat 
V_1,\ldots,\widehat V_p$, forms a set of defining loops for $\widehat G$ on $\widehat S^+$.

We note that there are exactly $q$ equivalence classes of parabolic fixed points in $\widehat G$, one for each of the loops $L_i$. Assume that $V_1,\ldots,V_p$ is not a valid set of defining loops for $L_1,\ldots L_q$. Then there is a lift of some $\widehat V_i$ that passes through a parabolic fixed point more than once. It follows that the corresponding loop $V_i$ must cross the corresponding geodesic $L_j$ more than once. Further, one sees, by opening up the parabolic fixed point into a hyperbolic transformation and looking at the lift of the corresponding $L_j$ as being a short path between the fixed points, that, in this case, there are necessarily two points of intersection of $V_i$ with $L_j$, and that there is an arc of $L_j$ between these two points, so that, not only does $V_i$ lift to a loop in $\Omega$, but both loops formed by cutting and pasting $V_i$ with this arc of $L_j$ lift to loops in $\Omega$. 

\begin{Prop}There is at least 
one valid set of defining loops $V_1,\ldots,V_p$, for every set of pinchable 
geodesics, $L_1,\ldots,L_q$. \end{Prop} 
\begin{proof} 
Suppose $V_1,\ldots,V_p$ is a set of defining loops on $S$ so that the 
corresponding set of defining loops, $\widehat V_1,\ldots,\widehat V_p$, on 
$\widehat S$ is not valid. Then there is a lift, $\tilde L_i$, of some $L_i$, 
and there is a lift $\tilde V_j$ of some $V_j$, so that $\tilde L_i$ and $\tilde 
V_j$ cross more than once. We construct a new loop, call it $\tilde W$ in $\Omega(G)$ as follows. Start at a point of intersection $x$ of $\tilde L_i$ and $\tilde V_j$, and follow the loop $\tilde V_j$ in one of the two possible directions from $x$ until it again intersects $\tilde L_i$ at the point $y$; then follow $\tilde L_i$ back to $x$. Denote by $A$ the arc of $\tilde L_i$ between $x$ and $y$. Observe that if $A$ intersects some lifting of some $V_m$, then, since that lifting of $V_m$ is a loop that is disjoint from $\tilde V_j$, this lifting of $V_m$ must cross $A$ an even number of times. Since $A$ is a compact subset of $\Omega(G)$, it can have only a finite number of points of intersection with liftings of the defining loops. Hence, there is a pair of crossing points of 
$A$ with some lift $\tilde V_k$ of some $V_k$, so that there are no 
other crossing points of $A$ with any lift of any defining curve 
between these two. We now replace $\tilde V_k$ by cutting it at these two 
crossing points, and replacing it by following parallel paths on either side of $A$. After this cut and paste operation, we have replaced $\tilde V_k$ by either one or two loops, depending on whether the two crossings of $\tilde V_k$ with $A$ occur with the same or opposite orientations. 
We then replace $V_k$ by these one or two loops, and observe that we now have either $p$ or $p+1$ simple disjoint loops on $S$ that lift to loops; it is easy to observe that $p$ of these loops must be homologically independent.

If we have replaced $V_k$ by one loop, there is nothing further to be said. If we have replaced it by two loops, then we observe that we now have $p+1$ loops on $S$ defining the Schottky covering.That is, we now have $p+1$ simple disjoint loops, which lift to loops, where every loop that lifts to a loop is freely homotopic to some product of these. It is well known 
that one of these $p+1$ loops is necessarily redundant. In any case, we  have shown that we can find a new 
set of defining loops for this Schottky covering, where the total number of 
crossing points of the defining loops with the set of pinchable loops, 
$L_1,\ldots,L_q$, has been decreased by two. \end{proof}

We remark that the above operation replaces one set of defining loops by 
another, and so might change the generators of the Schottky group.

%%%%%%%%%%%%%%%%%%%%%%%%
\subsection{Sufficiently complicated pinchable sets of geodesics} 
For any valid set of defining loops, $V_1,\ldots,V_p$, 
we define the complexity as follows. 
First, we can deform the $V_i$ on $S$ so that they are all geodesics. Then 
the geometric intersection number, $V_i\bullet L_j$, of $V_i$ with $L_j$ is 
well defined; it is the number of points of intersection of these two 
geodesics. Looking on the corresponding noded surface $\widehat S^+$, 
$V_i\bullet L_j$ is the number of times the curve $\widehat V_i$ obtained from $V_i$ by 
contracting $L_j$ to a point, passes through that point (node).
The {\em complexity} of $V_1,\ldots,V_p$ then is 
$$\max\sum_{i=1}^p V_i\bullet L_j,$$ 
where the maximum is taken over all $j=1,\ldots,q$.
The {\em complexity} $\Xi(L_1,\ldots,L_q)$ is the 
minimum of the complexities of $V_1,\ldots,V_p$, 
where the minimum is taken over all valid sets of defining loops.
The crucial point of the complexity is that if $\Xi(L_1,\ldots,L_q)\ge n$, then, 
for every valid defining set $V_1,\ldots,V_p$, there is a node 
$P$ on $S^+$ so that the total number of crossings of $P$ by $\widehat 
V_1,\ldots,\widehat V_p$ is at least $n$.
Now assume that $q=3p-3$, so that $\widehat G$ is a maximal noded Schottky group. 
Observe that $\widehat G$ is rigid, and that every part of $S^+$ is a sphere with three distinct nodes. Also, every connected component 
$\Delta\subset\widehat\Omega$ is a 
Euclidean disc $\Delta$, where $\Delta/\Stab(\Delta)$ is a sphere with three punctures; the three punctures correspond to the three nodes of the corresponding part of $\widehat S^+$. Let $V_1,\ldots,V_p$ be a valid set of defining loops on $S$, 
and let $\widehat V_1,\ldots,\widehat V_p$ be the corresponding loops on 
$\widehat S^+$. For each $i=1,\ldots,p$, 
the intersection of a lifting of $\widehat V_i$ with a 
component of $\widehat G$ (i.e., a connected component of $\widehat\Omega$) is 
called a {\em strand} of that lifting $\widehat V_i$. 
Similarly, the loops $\widehat V_1,\ldots,\widehat V_p$ 
appear on the corresponding parts 
of $\widehat S^+$ as collections of {\em strands} connecting the nodes on the boundary 
of each part. There are two possibilities for these strands; either a strand connects two distinct nodes on some part, or it starts and ends at the same node.

Since the loops $V_1,\ldots,V_p$ are simple and disjoint, 
there are at most three sets of 
{\em parallel} strands of the $\widehat V_i$ in each part; that is, there 
are at most three sets of strands, where any two strands in the same set are homotopic 
arcs with fixed endpoints at the nodes. We regard each of these sets of strands on a single part as being a {\em superstrand}, so that there are at most $3$ superstrands on any one part.

We next look in some component $\Delta$ of $\widehat G$, and look at a parabolic fixed point $x$ on its boundary, where $x$ corresponds to the node $N$ on the part $S_i$ of $\widehat S^+$. In general, there will be infinitely many liftings of superstrands emanating from $x$ in $\Delta$, but, modulo $\Stab(\Delta)$. there are only finitely many. In fact, there are at most $4$ such liftings of superstrands emanating from $x$. If there is exactly one superstrand on $S_i$ with one endpoint at $N$, and the other endpoint at a different node, then modulo $\Stab(\Delta)$ there will be exactly the one lifting of this superstrand emanating from $x$. If there is only one superstrand on $S_i$ with both endpoints at the same node $N$, then this superstrand has two liftings starting at $x$, one in each direction; so, in this case, we see two lifts of superstrands modulo $\Stab(\Delta)$ emanating from $x$. It follows that, modulo $\Stab(\Delta)$, we can have $0$, $1$, $2$, $3$ or $4$ liftings of superstrands starting at $x$. We note that these liftings of superstrands all end at distinct parabolic fixed points on the boundary of $\Delta$.

The defining set of loops, $\widehat V_1,\ldots,\widehat V_p$ is 
{\em sufficiently complicated} if there are two (different) lifts $\widehat 
C_i$ and $\widehat C_j$, of some $\widehat V_i$ and some not necessarily 
distinct $\widehat V_j$, respectively, so that 
$\widehat C_i$ and $\widehat C_j$ both pass through the parabolic fixed point 
$z_1$, into a component $\Delta_1$ of $\widehat G$, then both travel through 
$\Delta_1$ to the same parabolic fixed point on its boundary, $z_2$, and into 
another component $\Delta_2$, which they again traverse together to the same 
boundary point, $z_3$, necessarily a parabolic fixed point, where they enter 
$\Delta_3$, and they leave $\Delta_{3}$ at different parabolic fixed points.

%%%%%%%%%%%%%%%%%%%%%%%%
\subsection{Sufficiently complicated noded Schottky groups}
A noded Schottky group $\widehat G$ is {\em sufficiently complicated} if every set of valid 
defining loops on $\widehat S^+$ is sufficiently complicated.
We note that (keeping the notation of last subsection), inside $\Delta_1$, $\widehat C_i$ and $\widehat C_j$ are disjoint; 
they both enter $\Delta_1$ at the same point, and they both leave $\Delta_1$ 
at the same point; hence they cannot both be circles. This idea will be 
explored further in Sect. \ref{sec:genus3}. However, we need the full strength 
of the hypothesis of sufficiently complicated to prove 
that no nearby Schottky group can be classical.

\begin{Prop}\label{prop:11} If a maximal noded Schottky group $\widehat G$ has 
complexity at least $11$, then it is sufficiently complicated. \end{Prop}

\begin{proof} Let $V_1,\ldots,V_p$ be a valid set of defining loops for $\widehat G$. 
Then, since the complexity $n\ge 11$, there is a node 
$N\subset \widehat S^+$ that is crossed at least $n$ times by $\widehat 
V_1,\ldots,\widehat V_p$.
Let $S_1$ and $S_2$ be the not necessarily distinct spheres with three nodes 
on either side of the node $N$. 
Consider the finite set of $n$ strands in $S_1$, starting at the node $N$, 
defined by $\widehat V_1,\ldots,\widehat V_p$; call these strands 
$A_1,\ldots,A_n$. Ignoring direction, they separate into at most three superstrands (i.e., sets 
of parallel strands), each with a distinct endpoint; however, one of these superstrands
might start and end at the same endpoint, $N$.
Normalize the universal covering of $S_1$ so that 
it is the upper half-plane $\Hw$, 
and so that $N$ is the projection of the point at infinity; assume that the 
stabilizer of $\infty$ is generated by $P(z)=z+1$. Choose some lifting 
$\widetilde A_{1,0}$ of $A_1$, starting at $\infty$, and consider the set $\widetilde A$ of 
all liftings of all the $A_i$ starting at $\infty$, and lying between $\widetilde 
A_{1,0}$ and its image under $z\mapsto z+1$. As observed above, we can choose $\widetilde A_{1,0}$ so that  $\widetilde A$ will divide into 
at most four superstrands, where each of these four superstrands has a distinct endpoint 
on the circle at infinity of $\Hw$. We call these four superstrands, $\widetilde 
A_1,\ldots,\widetilde A_4$; we will give an explicit ordering below.
We write the universal covering of $S_2$ as $\widetilde \Hw$, which we can think of 
as being the set $\Im(z)<\alpha <0$, with $\infty$ projecting to $N$. (Note that even if $S_1=S_2$, following the lifting of a path through the node $N$, we travel from one representation of the universal covering of $S_1$ to another.)  Each of 
the strands in $S_1$ has a unique continuation in $S_2$ so that, when we look 
in $\Hw\cup\widetilde \Hw$, we see a set of $n$ arcs, each starting at some point 
on the real axis, passing through infinity, and ending at some point on the 
boundary of $\widetilde \Hw$; the crucial point is that these arcs touch without 
crossing at $\infty$.
For each $i=1,\ldots,4$, the continuations of the strands of $\widetilde A_i$ separate 
into some number, call it $m_i$, of superstrands in $\widetilde \Hw$. 
As above, $m_i\le 4$. We now order the $\widetilde A_i$ so that 
$m_1\ge\cdots\ge m_4$.
It is easy to observe that if $m_1=4$, then, for $i=2,3,4$, $m_i\le 2$. 
Similarly, if $m_1=3$, then $m_2\le 3$, $m_3\le 2$ and $m_4\le 2$. In any 
case, we have that $m_3\le 2$ and $m_4\le 2$.
Now let $|\widetilde A_i|$ be the number of strands in $\widetilde A_i$; we are given that 
$|\widetilde A_1|+|\widetilde A_2|+|\widetilde A_3|+|\widetilde A_4|=n\ge 11$.
This set of defining loops is sufficiently complicated if 
there are two strands of some $\widetilde A_i$ whose continuations in $\widetilde \Hw$ 
are also parallel. We now assume that this set of defining loops is not 
sufficiently complicated. There are several cases to consider.
We first assume that $m_1=4$; then $m_j \leq 2$, $j=2,3,4$, from which 
it follows that $|\widetilde A_1|=4$, $|\widetilde A_2|\le 2$, 
$|\widetilde A_3|\le 2$ and $|\widetilde A_4|\le 2$, so that 
$|\widetilde A_1|+|\widetilde A_2|+|\widetilde A_3+|\widetilde A_4|\le 10$.
If $m_1=3$, then $m_2\le 3$, $m_3\le 2$ and $m_4\le 2$. 
Then $|\widetilde A_1|=3$, $|\widetilde A_2|\le 3$, $|\widetilde A_3|\le 2$ 
and $|\widetilde A_4|\le 2$, so that $|\widetilde A_1|+|\widetilde 
A_2|+|\widetilde A_3|+|\widetilde A_4|\le 10$.
If $m_1=2$, then $m_2\le 2$, $m_3\le 2$ and $m_4\le 2$, 
from which it follows, as above, that 
$|\widetilde A_1|+|\widetilde A_2|+|\widetilde A_3|+|\widetilde A_4|\le 8$.
Finally, if $m_1=1$, then 
$|\widetilde A_1|+|\widetilde A_2|+|\widetilde A_3|+|\widetilde A_4|\le 4$. 
\end{proof}

\begin{Prop} For each positive integer $n$, and for each fixed genus $p\ge 2$, 
there are only finitely many topologically distinct maximal noded Schottky 
groups of genus $p$ and complexity $n$. \end{Prop} 
\begin{proof} Fix a Schottky group, with defining loops, $C_1,\ldots,C'_p$, with a set of pinchable curves 
on the underlying surface, $S$. Let $V_1,\ldots,V'_p$ be the projections of the defining loops to $S$, and let $L_1,\ldots,L_q$ be a set of pinchable loops, where the number of points of intersection of each $V_i$ with each $L_j$ is the geometric intersection number. We assume that these loops realize the complexity; that is, $\sum_{j=1}^pV_1\bullet L_j=n$, and, for every $i>1$, $\sum_{j=1}^pV_i\bullet L_j\leq n$. 
Consider all lifts of the $V_i$ to the 
fundamental domain defined by the defining curves; and contract the defining 
curves, $C_1,\ldots,C'_p$, to points. This yields a planar graph with $2p$ vertices and $m$ 
edges, where $$m=\sum_i\sum_j V_i\bullet L_j.$$ Since the genus of the underlying surface is $p$, there are at most $3p-3$ pinchable curves, from which it follows that $m\leq (3p-3)n$. It is clear that if two such 
planar graphs, with the same pairing of the vertices, including the 
identifications of the edges at each pair of vertices, are topologically 
equivalent, then the corresponding noded Schottky groups are topologically 
equivalent.
Our result now follows from the obvious fact that there are only finitely many 
topologically distinct planar graphs with $2p$ vertices and $m$ edges. 
\end{proof}

\begin{Prop}\label{prop:infmax} For every genus $p\ge 2$, there are infinitely 
many topologically distinct geometrically finite maximal noded Schottky groups 
of genus $p$.\end{Prop}
\begin{proof} Let $S_0$ be a surface of genus $2$, and 
let $A_1,B_1,A_2,B_2$ be a standard homotopy basis on $S_0$. Let $N$ be the 
smallest normal subgroup of $\pi_1(S_0)$ containing $A_1$ and $A_2$. Then we 
can regard the corresponding Schottky group $G_0$ as being generated by the 
elements $b_1$ and $b_2$, corresponding to $B_1$ and $B_2$, respectively. We 
first replace $B_1$ and $B_2$ by their corresponding geodesics. We then 
observe that the geodesic corresponding to the element 
$(B_1B_2)^nB_1A_1B_1\ngn A_1\ngn(B_2B_1)^{-n}$ is simple and disjoint from the 
geodesics corresponding to $B_1$ and $B_2$.
The corresponding word $W_n$ in the free group generated by 
$b_1$ and $b_2$ is then $w_n=(b_1b_2)^n(b_2b_1)^{-n}$. This word is obviously 
cyclically reduced and primitive. Hence the triple, $(b_1,b_2,w_n)$, is 
pinchable. Since the length $|w_n|=4n$, these triples are all distinct. Since 
$b_1$ and $b_2$ are generators, and $w_n$ cannot be a free generator, no 
element of $Aut({\mathcal F}_2)$ can map the triple $(b_1,b_2,w_n)$ onto the 
triple $(b_1,b_2,w_m)$ unless $m=n$. 
For genus $2$, the desired result follows at once; for higher genus, 
the result follows from Proposition \ref{prop:addpinch}. 
\end{proof} 
\begin{Prop} For every genus $p\ge 2$, there are infinitely many sufficiently 
complicated maximal noded Schottky groups of genus $p$. \end{Prop} 
\begin{proof} This is an immediate consequence of the above three propositions. 
\end{proof} 
\begin{Prop}\label{prop:finmax} There are only finitely many topologically distinct maximal 
neoclassical noded Schottky groups. 
\end{Prop}
\begin{proof} 
Let $C_1,\ldots,C'_p$ be defining circles for a maximal neoclassical 
noded Schottky group. Each of the $2p$ circles can intersect any of the other 
circles in at most one point, so there is a bound on the number of parabolic 
fixed points on the boundary of the fundamental domain defined by the common 
exterior of these circles. Except for the fact that we cannot tell the 
difference between an element and its inverse, the parabolic element fixing 
each of these points can be uniquely written as a word $W$ in the generators, 
$a_1,\ldots,a_p$ (see \cite{bm:combiv}).
One can describe this word $W$ as follows. Start with a point of 
tangency $x=x_1$ of say $C_i$ and $C_j$, then $a_j$ maps $x_1$ to a point of 
tangency, $x_2=a_j(x_1)$, of $C'_j$ with some $C_k$ or $C'_k$. If it is $C_k$, 
then $a_k(x_2)=a_ka_j(x)$ is a point of tangency of $C'_k$ with some other 
circle; if it is $C'_k$, then $a_k\ngn(x_2)=a_k\ngn a_j(x)$ is a point of 
tangency of $C_k$ with some other circle. Continuing in this manner, we 
eventually return to $x$, and so generate the word $W=\ldots a_ka_j$, or 
$W=\ldots a_k\ngn a_j$, which fixes $x$.
We remarked above, and it was shown in \cite{bm:combiv}, that every 
parabolic element of $G$ is a power of a conjugate of one of the parabolic 
elements whose fixed point lies on the intersection of two of the defining 
loops.
Since any two of the $2p$ circles, $C_1,\ldots,C'_p$ has at most 
one point of tangency, the word $W$ can contain a juxtaposed pair of 
generators $a_ia_j$, or $a_i\ngn a_j$, or $a_ia_j\ngn$, or $a_i\ngn a_j\ngn$, 
at most once. Hence, for any neoclassical noded Schottky group $G$, there is 
a bound on the lengths of these words defining the pinched elements of $G$. It 
follows that there is a finite list, $G_1,\ldots,G_k$, of maximal neoclassical 
groups so that if $G$ is any maximal neoclassical group, then there is a $G_i$ 
in our list, and there is an isomorphism $\phi:G\to G_i$ with the property 
that $g\in G$ is parabolic if and only if $\phi(g) \in G_i$ is parabolic. 
However, it was shown in \cite{K-M-S:circlepack} that every such isomorphism 
is in fact a conjugation by either an element of $\PSLC$, or by an orientation 
reversing conformal homeomorphism of $\Ch$. \end{proof}

\s 
\noindent 
{\bf Remark.} The hypotheses of the above proposition are too strong; 
we do not need to assume that the noded Schottky group is maximal. To see 
this, observe that, since any pair of distinct circles can meet in at most one 
point, there are only finitely many topologically distinct configurations of 
$2p$ circles, all having a common outside. Also, for any two such sets of 
$2p$ circles having the same configuration of points of tangency, it is easy 
to construct a path of topological deformations from one to the other. Hence 
it suffices to choose one such set of $2p$ circles for each such configuration.
Since we require that every point of tangency be a parabolic fixed point, 
it follows that each generator, $\widehat a_i$, maps the points of tangency 
on $\widehat C_i$ to the points of tangency on $\widehat C'_i$. Finally, we 
observe that, for each $i$, and each possible choice of pairings of the 
tangency points on $\widehat C_i$ with the corresponding tangency points on 
$\widehat C'_i$, one of the following situations occurs. Either there are at 
least $3$ points of tangency on $\widehat C_i$, in which case the 
transformation $\widehat a_i$ is determined by the pairings of these points of 
tangency; or there are exactly two points of tangency on $\widehat C_i$, in 
which case, $\widehat a_i$ is determined up to a $1$-parameter continuous 
family of hyperbolic motions keeping these two points fixed and keeping 
$\widehat C_i$ invariant; or there is exactly one point of tangency on 
$\widehat C_i$, in which case $\widehat a_i$ is determined up to a 
$1$-parameter continuous family of parabolic motions keeping this one point 
fixed and keeping $\widehat C_i$ invariant; or there are no points of 
tangency, in which case $\widehat a_i$ is determined up to an arbitrary 
element of the $3$-parameter subgroup of $\PSLC$ keeping both discs bounded by 
$\widehat C_i$ invariant.

%%%%%%%%%%%%%%%%%%%%%%%
%%%%%%%%%%%%%%%%%%%%%%%
\section{Cross ratios and inequalities for sufficiently complicated groups} 
If $z_1$, $z_2$, $z_3$, $z_4$ are four distinct points on the extended 
complex plane, $\Ch$, then we denote their cross-ratio by 
$$(z_1,z_2;z_3,z_4)=\frac{(z_1-z_3)(z_2-z_4)}{(z_1-z_4)(z_2-z_3)}.$$ We will 
use the well known facts that elements of $\PSLC$ preserve the cross-ratio, 
and that the cross-ratio is real if and only if the four points lie on a 
circle. We will usually use the cross-ratio in the normal form $$(0,\infty;z,-i)=iz.$$
We will need the following two easy exercises concerning the level $2$ 
congruence subgroup of the modular group.

\begin{Prop} \label{prop:distance} 
Let $H$ be a Fuchsian group representing the thrice punctured sphere. 
We assume that $H$ acts on the upper half-plane, and that $H$ has been 
conjugated so that $z\mapsto z+\alpha $, $\alpha\ge 1$, is a primitive 
parabolic element in $H$. Let $E_1$ and $E_2$ be two complete geodesics 
starting at $\infty$, where $E_1$ and $E_2$ end at distinct parabolic fixed 
points $y_1$ and $y_2$ of $H$, and where $E_1$ and $E_2$ project onto complete 
simple disjoint geodesics on $\Hw/H$. Then $|y_1-y_2|\ge \frac{|\alpha |}{4}$. 
\end{Prop}
\begin{proof} Since $E_1$ and $E_2$ are disjoint geodesics, either $y_1$ 
or $y_2$ is not $H$-equivalent to $\infty$; we can assume it is $y_1$. 
Conjugate $H$ by a translation so that $y_1=0$. Observe that $H$ contains the 
primitive transformations $z\mapsto z+\alpha$ and $z\mapsto\frac{z}{\pm \beta 
z+1}$, where $\beta=\frac{-4}{\alpha }$. Then, since the projection of $E_2$ 
cannot cross the projection of $E_1$, $|y_2|\ge\frac{\alpha }{4}$. \end{proof}

For the next proposition, we consider $\widehat G_0$ as being a maximal 
noded Schottky group, and we consider the truncated flat part of 
$B^{\alpha ,n}$ as defined in Sect. \ref{sec:vertical}. If $\Delta \subset \Omega(\widehat G_{0})$, then, for each $\alpha $ and each $n$,  we call the intersection of the
truncated flat part of $B^{\alpha ,n}$ with $\Delta$ the $(\alpha ,n)$-{\em compact part} of $\Delta$. Up to conjugation, we always may assume that $\Delta$ is the
upper half-plane.

\begin{Prop}\label{prop:slope} 
Let $H$ be a Fuchsian group, acting 
on $\Hw$ and representing the thrice punctured sphere, where $H$ is generated 
by $a(z)=z+\alpha $ and $b(z)=\frac{z}{\beta z+1}$; where $\alpha\ge 1$ and $\alpha \beta=-4$. 
Let $R$ be the Euclidean ray through the origin defined by $\arg(z)=\theta$, 
$0<|\theta|\le \pi/2$. If $|\theta|\le \pi/6$, then $R$ does not project to a 
simple path on $\Hw/H$. Further, given $K>0$ and an $(\alpha 
,n)$-compact part of the upper half-plane, there is a constant $\theta_0$, 
$0<\theta_{0}<\pi/6$, so that, for all $\theta$, with 
$0<|\theta|<\theta_{0}$, there are four points, $z_1$, $z_2$, $z_3$, $z_4$, in 
the $(\alpha ,n)$-compact part of the upper half-plane, so that 
$$|\Im(z_1,z_2;z_3,z_4)|\ge K.$$ 
\end{Prop} 

\begin{proof} For the first statement, conjugate $H$ by a dilation, 
so that $\alpha=-\beta=2$. Consider the transformation 
$(ab)\ngn(z)=\frac{z-2}{2z-3}$, and observe that this maps $R$ onto a set disjoint 
from $R$ precisely when $|\tan\theta|>\frac{1}{\sqrt{3}}$. 
The second statement follows almost at once from the observation that, 
as $\theta\to 0$, the region between the line $\arg(z)=\theta$, and its 
translate under $ba$ increases, and covers larger and larger portions of any 
compact part of the upper half-plane. \end{proof}

\begin{Prop}\label{prop:suffcomp} Let $\widehat G$ be a sufficiently 
complicated maximal noded Schottky group. Then, for any set of defining 
loops, $\widehat C_1,\ldots,\widehat C'_p$, there are four successive parabolic 
fixed points, $z_1,z_2,z_3,z_4$ on some $\widehat C_i$, so that 
$|\Im(z_2,z_3;z_4,z_1)|\ge \frac{1}{8}$. \end{Prop}
\begin{proof} Let $\widehat C_1,\ldots,\widehat C'_p$ be any set of defining 
loops for $\widehat G$. Since $\widehat G$ is sufficiently complicated, there is some 
$\widehat C_j$, call it $\widehat C$, and there is some translate $\widehat C'$ 
of some $\widehat C_k$, and there are three distinct components 
$\Delta_1$, $\Delta_2$ and $\Delta_3$ of $\widehat G$, so that, after 
appropriate normalization, $\widehat C$ and $\widehat C'$ both enter 
$\Delta_1$ at the parabolic fixed point $-i=z_1$; these curves both leave 
$\Delta_1$ and enter $\Delta_2$ at the parabolic fixed point $0=z_2$; they 
both leave $\Delta_2$ at the parabolic fixed point, $\infty=z_3$, where they 
enter $\Delta_3$; and they leave $\Delta_3$ at distinct parabolic fixed 
points, $z_4$ and $z'_4$. 
We remark that, while $\widehat C$ and $\widehat C'$ are required to be distinct, 
$j$ is not necessarily distinct from $k$.

Since $\Delta_2$ and $\Delta_3$ are disjoint, and both have $\infty$ on their boundary, their boundaries in ${\mathbb C}$ consist of two parallel lines; the boundary of $\Delta_2$ passes through the origin. It then follows that $\Delta_1$, which is disjoint from both $\Delta_2$ and $\Delta_3$ is a Euclidean disc of finite diameter; let $d$ denote the diameter of $\Delta_1$. 

Let $g(z)= z+\alpha$ be a generator of $\Stab(\infty)\subset\Stab(\Delta_2)$; 
then $g$ is also a generator of $\Stab(\infty)$ in $\Stab(\Delta_3)$. Since 
$g(\Delta_1)\ne\Delta_1$, we must have that 
$g(\Delta_1)\cap\Delta_1=\emptyset$; hence $|\alpha|\ge d$.
We now write $\Delta_2$ as $\{z|-\pi/2+\theta<\arg(z)<\pi/2+\theta\}$; or, 
equivalently, \{$\Re(e^{-i\theta}z)>0\}$. Since $-i$ is not on the boundary of 
$\Delta_2$, $0<\theta< \pi$.
The boundary of $\Delta_3$ is a line parallel to the boundary of $\Delta_2$ 
and passing through the imaginary axis at some point $y<-1$. Since the 
boundary of $\Delta_1$ is a Euclidean circle passing through the point $-i$ 
and tangent to $\arg(z)=\pi/2+\theta$ at the origin, its diameter is 
given by $d=\csc\theta$.
Replacing $g(z)=z+\alpha $ by its inverse if necessary, we can write $\alpha 
=\rho e^{i(\theta-\pi/2)}$, where $|\alpha |=\rho \geq d=\csc\theta$. We observe 
that $\Re(\alpha )=|\alpha |\sin\theta\ge 1$. We conclude from Proposition 
\ref{prop:distance} that $|z_4-z'_4|\ge \frac{|\alpha |}{4}$. It follows that 
$$|\Re(z_4)-\Re(z'_4)|=|z_4-z'_4|\sin\theta\ge\frac{|\alpha| 
\sin\theta}{4}\ge\frac{1}{4}.$$ Our conclusion now 
follows from the fact that $\Im(z_2,z_3;z,z_1)=\Im(0,\infty;z,-i)=\Re(z)$. 
\end{proof}

%%%%%%%%%%%%%%%%%%%%%
%%%%%%%%%%%%%%%%%%%%%
\section{Non-classical Schottky groups} Let $\widehat C_1,\ldots,\widehat C'_p$, 
with generators, $\widehat a_1,\ldots,\widehat a_p$, be a set of defining loops for 
the sufficiently complicated maximal noded Schottky group $\widehat G$, and let 
$G_n$ be a sequence of Schottky groups converging to $\widehat G$ in a relative conical 
neighborhood (see Sect. \ref{sec:vertical}). That is, there is some number 
$\alpha_0$, so that the relative conical neighborhood consists of all Schottky groups 
$G^{\alpha}$ obtained via vertical projection from 
$\widehat C_1,\ldots,\widehat C'_p$, where the vertical projection is defined 
with reference to the infinite shoebox of size $\alpha$. We reiterate that the 
infinite shoebox, and hence the vertical projection, depend on a choice of 
conjugating transformations, one for each conjugacy class of maximal parabolic 
subgroups of $\widehat G$. 

\begin{Thm} Let $\widehat G$ be a sufficiently complicated maximal noded 
Schottky group. If $\alpha_0$ is sufficiently large, then every Schottky 
group in the relative conical neighborhood of $\widehat G$ defined by $\alpha _0$ is 
not classical. 
\end{Thm}
\begin{proof} We fix some $\alpha_0$ so that, using our normalizing 
transformations, we have the proper precise invariance; that is, the 
complement of the flat part $B_0$, defined by $\alpha_0$, is a disjoint union 
of pairs of tangent Euclidean discs, where each pair is precisely invariant under the 
parabolic cyclic subgroup with fixed point at the point of tangency. As above, 
we normalize $\widehat G$ so that $\infty$ is a point in the truncated flat part 
of $B_0$; then there is some component $\Delta_0$ of $\widehat G$ so that 
$\infty\in\Delta_0$. 
Since $\widehat G$ is maximal noded, there are $3p-3$ parabolic elements, 
$\widehat g_1,\ldots,\widehat g_{3p-3}$, that generate non-conjugate maximal 
parabolic cyclic subgroups of $\widehat G$. For each $\alpha >\alpha_0$, set 
$g_i^{\alpha}=f^{\alpha}\widehat g_i(f^{\alpha})\ngn$. 
Fix $\epsilon>0$; we first choose $\alpha_1>\alpha_0$ so that, 
for $\alpha>\alpha_1$, the fixed points of each $g_i^{\alpha}$ lie within 
$\epsilon$ of the fixed point of $\widehat g_i$.
We need the following in order to discuss the twist of the projection of a defining 
curve near a node. For each fixed $i$, 
renormalize $\widehat G$ so that $\widehat g_i(z)=z+1$. For $m \in {\mathbb 
Z}$, we call the lines $\{\Re(z)=m\}$ {\em slope lines}. 
The truncated flat part of $B^{\alpha ,n}$ has non-trivial intersection with 
$2n+1$ slope lines (two of these are on the boundary). The slope lines 
are then well defined near every parabolic fixed point of $\widehat G$.

The content of Lemma \ref{lem:slopeint} below is that, for $\alpha $ 
sufficiently large, if $C$ is a circular defining loop for $G^{\alpha }$, then 
the vertical projection of $C$ has only a finite amount of twist as it passes 
through each node.
Now suppose that $G^{\alpha }$ 
is classical, where $\alpha >\alpha _2$, chosen as above. Let 
$C_1,\ldots,C'_p$ be defining circles for $G^{\alpha }$, and let $\widehat 
C_1,\ldots,\widehat C'_p$ be their vertical projections, defined by $\alpha _0$. 
Also, let $n$ be as given in Lemma \ref{lem:slopeint}.
If $\widehat C_1,\ldots,\widehat C'_p$ is not a set of defining loops for 
$\widehat G$, then there is some $\widehat C_i$ that passes more than once through 
the same parabolic fixed point. In this case, it is clear that we can find 
four points on $\widehat C_i$, and lying in the truncated flat part of $B^{\alpha 
_0,n}$, so that the imaginary part of their cross-ratio is bounded away from 
zero, from which it follows that $C_i$ cannot be a circle.
If $\widehat C_1,\ldots,\widehat C'_p$ is a set of defining loops for $\widehat G$, 
then, since $\widehat G$ is sufficiently complicated, there is some defining curve 
$\widehat C$ so that $\widehat C$ enters some component $\Delta_1$ 
at the parabolic fixed point $x_1$, it leaves $\Delta_1$ and enters 
$\Delta_2$ at the parabolic fixed point $x_2$, it leaves $\Delta_2$ and enters 
$\Delta_3$ at the parabolic fixed point $x_3$, and it leaves $\Delta_3$ at the 
parabolic fixed point $x_4$, where, by Proposition \ref{prop:suffcomp}, 
$|\Im(x_2,x_3;x_4,x_1)|\ge\frac{1}{8}$.
If follows from Lemma \ref{lem:finitequad} below that there are in fact a 
finite number of sets of four parabolic fixed points, 
$(x_{1,1},x_{1,2},x_{1,3},x_{1,4}),\ldots (x_{N,1},x_{N,2},x_{N,3},x_{N,4})$, of 
$\widehat G$, so that the following holds.\hfil\break (i) For every 
$j=1,\ldots,N$, $$|\Im (x_{j,2},x_{j,3},x_{j,4},x_{j,1})|\ge \frac{1}{8},$$ 
and\hfil\break (ii) For every $\alpha >\alpha_2$, and for every classical set of 
defining loops $C_1,\ldots,C'_p$ of any $G^{\alpha}$, there is some $j$, $1\le 
j\le N$, and there is some $C_i$, so that $\widehat C$, the vertical projection 
of this $C_i$, passes through the four points 
$(x_{j,2},x_{j,3},x_{j,4},x_{j,1})$.
We next fix $n$ so that if $C$ is a circular defining loop for some 
$G^{\alpha }$, with $\alpha >\alpha _0$, then the vertical projection of $C$ 
crosses at most $n$ slope lines. Note that the minimal $n$ satisfying this 
requirement decreases as $\alpha _0$ increases; hence we can increase $\alpha 
_0$ while leaving $n$ constant.
We next choose $\alpha _0$ sufficiently large, so that, for each one of our 
finite number of sets of four parabolic fixed points, if $z_i$ is 
a point on the boundary of the flat part of $B^{\alpha _0,n}$, where $z_i$ 
lies on the horizontal part of the boundary component near $x_i$, then 
$|\Im(z_2,z_3;z_4,z_1)|\ge\frac{1}{16}$.
Finally, we choose $\alpha _1>\alpha _0$ so that for all $\alpha >\alpha _1$, 
$f^{\alpha }$ is sufficiently close to the identity so that, for all $z_1$, 
$z_2$, $z_3$, $z_4$ as above, $|\Im(f^{\alpha }(z_2),f^{\alpha 
}(z_3);f^{\alpha }(z_4),f^{\alpha }(z_1)|\ge\frac{1}{32}$. 
\end{proof}

\begin{Lem}\label{lem:slopeint} There is an $n\in{\mathbb Z}$, 
and an $\alpha _2>\alpha _1$, so that for any $\alpha 
>\alpha _2$, and for any defining loop $C$ for $G^{\alpha }$, 
where $C$ is a Euclidean circle, the following holds. 
Let $\widehat g$ be a parabolic element of $\widehat 
G$, where the vertical projection of $C$ passes through the fixed point of 
$\widehat g$. Then, near the fixed point of $\widehat g$, in the truncated flat part 
of $B^{\alpha _0,n}$, the vertical projection of $C$ passes through fewer than 
$n$ slope lines. 
\end{Lem} 
\begin{proof} We have already observed that $f^{\alpha }$ converges to the 
identity uniformly on compact subsets of the truncated flat part of 
$B^{\alpha _0,n}$. Then, for $\alpha $ sufficiently large, since $C$ is a 
Euclidean circle, the intersection of $(f^{\alpha })\ngn(C)$ with the 
truncated flat part of $B^{\alpha _0,n}$ is close to a straight line in the 
spherical metric. However, the content of Proposition \ref{prop:slope} is that 
a Euclidean straight line passing through the upper half-plane that projects 
to a simple arc on $\Hw/\Gamma_0$ passes through at most finitely many slope 
lines in the truncated flat part. The desired result now follows. 
\end{proof}

\begin{Lem}\label{lem:finitequad} There is a finite set of quadruples of 
parabolic fixed points, $(x_{1,1},x_{1,2},x_{1,3},x_{1,4}),$ $\ldots, 
(x_{N,1},x_{N,2},x_{N,3},x_{N,4})$, so that if $\widehat C$ is the vertical 
projection of a circular defining curve for some $G^{\alpha}$, where $\widehat 
C$ is a defining loop for $\widehat G$ passing through at least four parabolic fixed 
points, then there is a $g\in G$, and there is a $j$, $1\le j\le N$, so that 
$(x_{j,1},x_{j,2},x_{j,3},x_{j,4})$ are four successive parabolic fixed points 
on $g(\widehat C)$.
\end{Lem} 
\begin{proof}
Choose four successive fixed points $(x_1,x_2,x_3,x_4)$ on $\widehat C$. 
It is clear that we can choose $x_2$ to be the fixed point of one of 
$\widehat g_1,\ldots,\widehat g_{3p-3}$. Then $x_1$ lies in one of the two 
components stabilized by this $\widehat g_i$; call it $\Delta_1$.
As in Proposition \ref{prop:distance}, we note that, modulo the action of 
$\Stab(x_2)$, there are only finitely many parabolic fixed points on the 
boundary of $\Delta_1$, which can be reached by a path that projects to a simple 
path. We choose a four sided fundamental polygon $P_1$ for the action of 
$\Stab(\Delta_1)$ on $\Delta_1$, where $x_2$ lies on the boundary of $P_1$; 
using the action of $\Stab(x_2)$, we can find a translate of $\widehat C$ so 
that $x_1$ also lies on the boundary of $P_1$.
There is a natural continuation of the sides of $P_1$ into $\Delta_2$, and 
there is a natural unique fundamental polygon $P_2$ for $\Stab(\Delta_2)$ with 
these sides. The content of Lemma \ref{lem:slopeint} is that there there are a 
finite number of translates of $P_2$, under $\Stab(x_2)$, so that the endpoint 
of the intersection of this translate of $\widehat C$ with $\Delta_2$ lies on 
the boundary of one of these translates of $P_2$; i.e., given that the translate 
of $\widehat C$ passes through $x_1$ and $x_2$, there are only finitely many 
possibilities for $x_3$.
For each of the possible choices for $x_3$, we repeat the above argument to 
conclude that, since we can have only a finite amount of twist about $x_3$, 
there are only finitely many possibilities for $x_4$.
 \end{proof}

%%%%%%%%%%%%%%%%%%%%%
%%%%%%%%%%%%%%%%%%%%%
\section{Some noded Schottky groups that are not neoclassical}\label{sec:genus3} 
Consider the set of noded Riemann surfaces of genus 3, where each of these 
surfaces has exactly three nodes, all corresponding to dividing loops (this is 
a   complex 3-dimensional family of noded surfaces). In this section, we show 
that there is no neoclassical Schottky group uniformizing any of these noded 
surfaces. Then, in the next section, we pick a particular such noded surface, 
and add three more nodes, so that the resulting noded surface is highly 
symmetric. We choose a particular noded Schottky group representing this noded 
surface, and show that this noded Schottky group is sufficiently complicated, 
and hence not neo-classical.
We note that it has been conjectured that every closed Riemann surface can be 
uniformized by a classical Schottky group; some partial results along these 
lines are known. The moduli space of Riemann surfaces of genus $p$ has a natural 
compactification, the Deligne-Mumford compactification, obtained by adding noded 
Riemann surfaces. Our example here shows that the corresponding conjecture for 
noded Riemann surfaces does not hold; that is, there are noded Riemann surfaces 
with no neoclassical Schottky uniformizations.
Assume we are given a Schottky group $G$, and we have some set of defining 
loops, $C_1,C'_1, \ldots ,C_p,C'_p$, for $G$. Let $\mathcal D$ be the standard 
fundamental domain for $G$ bounded by these loops and denote by $P:\Omega(G) \to 
S=\Omega(G)/G$ the natural holomorphic covering. Assume that $R \subset S$ is a 
simple closed curve whose lift to ${\mathcal D}$ consists of pairwise disjoint 
simple arcs. Each connected component of $P^{-1}(R) \cap {\mathcal D}$ is called 
a {\it strand} of $R$. A strand $\tilde R$ of $R$ which has both its end points 
on the same defining loop on the boundary of $\mathcal D$ is called {\em 
essential} if it is not homotopically trivial in $\mathcal D$ relative to its 
boundary.
Consider a closed Riemann surface $S$ of genus $3$ with three 
homologically trivial, but homotopically distinct and non-trivial, mutually 
disjoint simple loops, which we think of as {\em red loops}; we label these red 
loops as $R_{1}$, $R_{2}$ and $R_{3}$. These are chosen so that 
$S-\{R_{1},R_{2},R_{3}\}$ consists of three one-holed tori, denoted by $P_{1}$, 
$P_{2}$ and $P_{3}$, and one 3-holed sphere, denoted by $P_{4}$; we assume 
that $R_{i}$ lies on the boundary of $P_{i}$, $i=1,2,3$ (see figure \ref{figure1}).

\begin{figure}
\centering
\includegraphics[width=7cm]{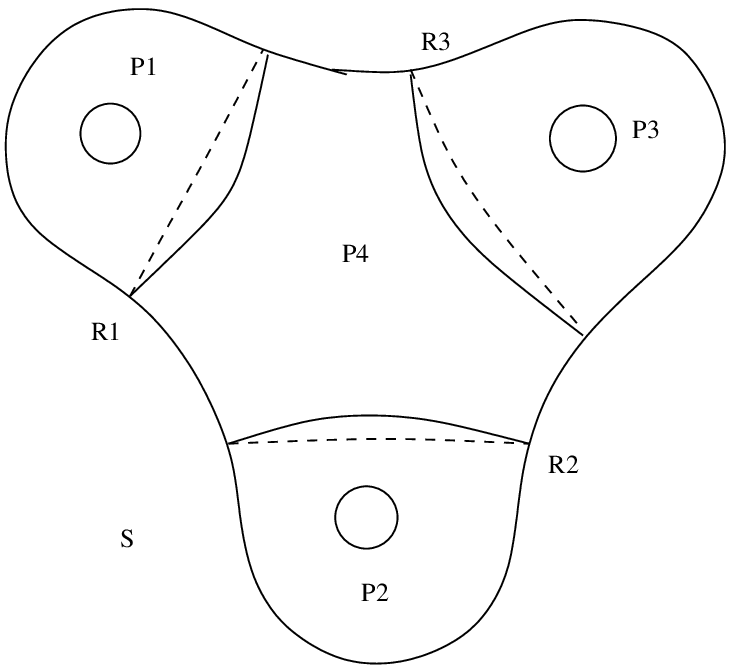}
\caption{}
\label{figure1}
\end{figure}

%%%%%%%%%%%%%
\subsection{Generators} Before proceeding, we give an explicit example showing 
that the noded surface $\widehat S$, obtained by shrinking all three red loops, 
can be uniformized by a noded Schottky group. We remark that it was shown by 
Hidalgo \cite{Hidalgo:NodedSchottky} that every noded Riemann surface can be 
represented by at least one noded Schottky group.
Let $A_1,B_1,A_2,B_2,A_3,B_3$ be a standard homology basis on some surface $S$ 
of genus $3$ as shown in figure \ref{figure4}. In the same figure are shown the oriented simple loops $R_{1}$, $R_{2}$, $R_{3}$, $W_{1}$, $W_{2}$, $W_{3}$. We have that $R_{j}$ is homotopically equivalent to $[A_{j},B_{j}]=A_{j}B_{j}A_{j}^{-1}B_{j}^{-1}$, for $j=1,2,3$. We also have that the loops $W_{1}$, $W_{2}$ and $W_{3}$ are respectively freely homotopic to $B_{3}^{-1}B_{1}^{-1}B_{2}^{-1}$,  $A_{2}A_{1}^{-1}$  and $A_{1}A_{3}^{-1}$. The three loops $W_{1}$, $W_{2}$ and $W_{3}$ are pairwise disjoint and homologically independent, in particular, they determine a Schottky covering (with corresponding Schottky group $G$)
$P:\Omega(G) \to S$, defined by the smallest 
normal subgroup $N\subset \pi_1(S)$ containing the elements, $w_1,w_2,w_3$, 
corresponding to the loops $W_1, W_2, W_3$.  Let us denote by $d_{1}, d_{2}, d_{3}$ generators of $G$ so that there is a fundamental set of loops for $G$, say $C_{1}$, $C'_{1}$, $C_{2}$, $C'_{2}$, $C_{3}$, $C'_{3}$, where $P(C_{j})=W_{j}$ and $d_{j}(C_{j})=C'_{j}$, for $j=1,2,3$. Projecting the elements corresponding to $R_1, R_2, R_3$ to $G=\pi_1(S)/N$, 
we obtain the corresponding elements $r_1=d_{1}d_{2}d_{3}^{-1}d_{1}^{-1}d_3d_2{-1}$, 
$r_2=d_{1}d_{2}d_{1}^{-1}d_{2}^{-1}$ and $r_3=d_{1}d_{3}^{-1}d_{1}^{-1}d_{3}$. These three words in 
the free group $G$ are obviously not powers and they obviously generate 
non-conjugate cyclic subgroups. Hence they are pinchable.\footnote{This 
particular example is closely related to some families of examples found by 
James Blumling and Vinitha Jacob, two high school students who worked with the 
second author.} We have shown that $\widehat S$ (the noded surface obtained from $S$ after pinching the loops $R_{1}$, $R_{2}$ and $R_{3}$) can be represented by a noded 
Schottky group.

\begin{figure}
\centering
\includegraphics[width=8cm]{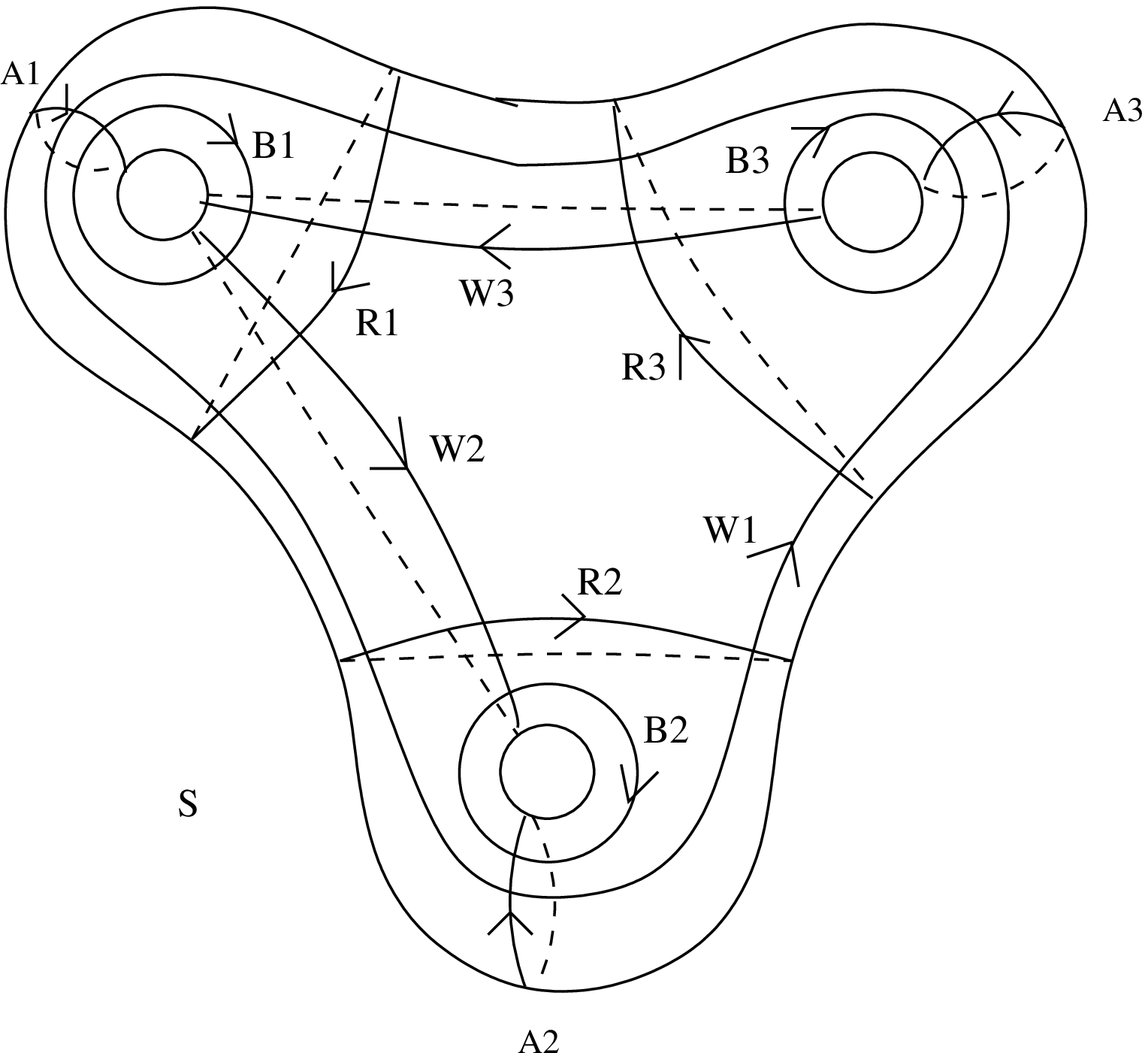}
\caption{}
\label{figure4}
\end{figure}

\begin{Thm} Let $\widehat S$ be some noded Riemann surface of genus $3$ formed 
by pinching the three loops, $R_1$, $R_2$ and $R_3$. Let $\widehat G$ be any 
noded Schottky group so that $\Omega^+(\widehat G)/\widehat G=\widehat S$. Then 
$\widehat G$ is not neoclassical. 
\end{Thm}

%%%%%%%%%%%%%%%%%%
\subsection{The planar graph} We assume that there is a neoclassical noded 
Schottky group $\widehat G$ representing $\widehat S$, and we assume that we 
are given some set of defining loops $\widehat C_1,\ldots,\widehat C'_3$ for 
$\widehat G$, where the $\widehat C_i$ are all Euclidean circles. We will not 
need the explicit generators pairing these. Let $G$ be some Schottky group in a 
relative conical neighborhood of $\widehat G$. Let $C_1,\ldots,C'_3$ be the defining 
loops for $G$ obtained via vertical projection, and let $R_1$, $R_2$, $R_3$ be 
the geodesics on $S=\Omega(G)/G$ corresponding to the three red loops. Let 
${\mathcal D}$ be the fundamental domain for $G$ bounded by $\widehat C_1, 
\ldots,\widehat C'_3$.
Consider the planar graph ${\mathcal G}$ obtained as follows. We shrink the six 
defining loops on the boundary of $\mathcal D$ to points, which become the 
vertices of ${\mathcal G}$; the edges of ${\mathcal G}$ are the strands of the 
liftings of $R_1$, $R_2$ and $R_3$ to ${\mathcal D}$. It is clear that 
${\mathcal G}$, as an embedded graph, is well defined up to a homeomorphism of 
$\Ch$.

\begin{Lem} No edge of ${\mathcal G}$ has both endpoints at the same vertex. 
\end{Lem} 
\begin{proof} 
If we had such an edge, starting and ending say at $C_i$, then the corresponding 
defining loop, $\widehat C_i$, would not be simple. \end{proof}
We assume, from here on, that ${\mathcal G}$ contains no edge with both endpoints 
at the same vertex.

\begin{Lem} Given any two vertices of ${\mathcal G}$, there is at most one edge 
connecting them. \end{Lem} 
\begin{proof} Suppose there were two edges starting at $C$ and end at $C'$, 
where $C$ and $C'$ are each one of loops $C_1,\ldots,C'_3$. Then, in 
$\Omega(G)$, there is an arc of the corresponding defining loop $C$ between 
these two edges, and there is likewise an arc of the corresponding defining loop 
$C'$. When we contract the red edges to points, these two arcs become arcs of 
the corresponding defining circles, $\widehat C$ and $\widehat C'$, which are 
now required to pass through the same two parabolic fixed points; this cannot 
be. 
\end{proof}

We assume, from here on, that $\mathcal G$ is such that there is at most one edge 
connecting any two vertices.
For each vertex $C_i$ (respectively, $C'_{i}$), we denote the number 
of edges of $\mathcal G$ ending at $C_i$ (respectively, $C'_{i}$) by $N_i$ 
(respectively, $N'_{i}$). Since $C_i$ and $C'_i$ are identified by an element 
of $G$, $N_{i}=N'_{i}$.
Since $N_{i}=N'_{i}$, $\sum N_i$ is the total number of edges of $\mathcal G$.
For $i=1,2,3$, let $V_{i}$ denote the projection of (the loop) $C_i$ to $S$. 
Observe that $V_1$, $V_2$, $V_3$ are three homologically independent simple 
disjoint loops on $S$. Then $N_i$ is the total number of points of 
intersection of $V_i$ with the dividing loops $R_1$, $R_2$ and $R_3$. Since the 
$R_j$ are dividing loops, each $N_i$ is even.
Similarly, let $M_j$ denote the sum of the number of crossings of $R_j$ with 
$V_1$, $V_2$ and $V_3$. Again, $M_j$ is necessarily even. It is clear that 
$M_j$ is the number of strands of liftings of $R_j$; hence $\sum M_j=\sum N_i$ 
is the total number of edges of $\mathcal G$.

\begin{Prop} For $j=1,2,3$, $M_j>2$. 
\end{Prop} 
\begin{proof} We cannot have $M_j=0$, for then the loop $R_j$ does not cross 
any of the $V_i$, so it lifts to a loop, which cannot be.
If $M_j=2$, then $R_j$ crosses exactly one of the $V_i$, say it is $V_1$, and 
it crosses $V_1$ twice. Hence the lift of $R_j$ defines exactly two essential 
strands. Both of these strands have one endpoint on $C_1$ and the other endpoint 
on $C'_1$, which also cannot be. 
\end{proof}

Since each $M_j\ge 4$, the total number of strands is at least $12$. However, 
since we do not have more than one edge from any one vertex to any other vertex, and 
each $N_i$ is even, we have that the total number of strands, which is the same 
as the total number of edges, is at most $12$. Hence the number of edges is 
exactly $12$, and there are exactly four edges emanating from each vertex.
It is easy to see that, up to isomorphism, there is exactly one graph 
with $6$ vertices; $4$ edges at each vertex; no edge having the same vertex 
at both ends; and no two edges having the same pair of vertices at their ends. 
We note that this graph is homogeneous, and that it has essentially only one 
embedding in the extended complex plane. We can think of the vertices as the 
points, $0$, $\infty$, $\pm 1$ and $\pm i$, and the edges as lying on the real 
and imaginary axes, and on the unit circle. Observe that this graph has as 
group of isometries the symmetric group on four letters, ${\mathcal S}_{4}$, 
generated by the M\"obius transformations $A(z)=iz$ and $B(z)=\frac{1-z}{1+z}$. 
We may think of the vertices as the midpoints of the faces of a cube; that 
is, $\mathcal G$ is the dual graph to the graph on ${\mathbb S}^2$ obtained 
from the regular cube by projection from the origin.
We have shown that the strands of $\mathcal G$ divide $D$ into $8$ regions, 
corresponding to the vertices of the cube. We have the natural projection 
from $D$ onto $S$, which is divided into the $4$ subsurfaces, $P_1$, $P_2$, 
$P_3$ and $P_4$, by the red loops. Hence we can identify each of the 
corresponding $8$ regions of $\mathcal D$ as projecting onto one of these four 
subsurfaces. Further, for every strand of $\mathcal G$, exactly one of the two 
sides of this strand projects onto the 3-holed sphere, $P_4$, and the other 
side projects onto one of the surfaces $P_{1}$, $P_{2}$ or $P_{3}$. It follows 
that $4$ of these regions project onto $P_4$, and that, for each of these 
regions projecting onto $P_{4}$, we must have that each of the $3$ neighboring 
regions projects onto a distinct $P_{j}$, $j=1,2,3$. It is easy to see that this 
is not possible (try to put a value $1$, $2$, $3$ or $4$ to each vertex of the 
cube so that (i) for any two adjacent vertices, one of the values is $4$, and 
(ii) at each vertex labeled $4$, all three values, $1$, $2$ and $3$, occur at an 
adjacent vertex).
This completes the proof of the fact that there is no neoclassical noded 
Schottky group representing any of our noded Riemann surfaces.

%%%%%%%%%%%%%%%%%%%%%%%%
%%%%%%%%%%%%%%%%%%%%%%%%
\section{A sufficiently complicated noded Schottky group}
 In this section, we 
construct a particular example of a sufficiently complicated maximally noded 
Schottky group of genus $3$. The construction starts with the surface $S_0$ of 
the last section, with the three red loops, $R_1$, $R_2$ and $R_3$ on it. We 
enlarge this set of simple disjoint loops by adjoining three non-dividing loops, 
which we think of as {\em green loops}, $G_1$, $G_2$ and $G_3$, where, for 
$i=1,2,3$, $G_{i}$ lies in the subsurface $P_i$ (see figure \ref{figure2}).

\begin{figure}
\centering
\includegraphics[width=8cm]{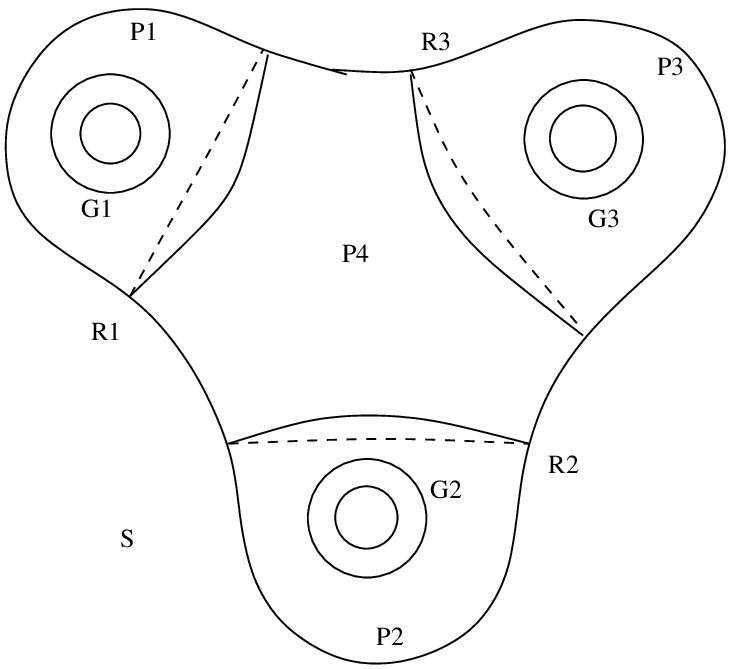}
\caption{}
\label{figure2}
\end{figure}

It was shown by
Hidalgo \cite{Hidalgo:NodedSchottky} that every noded Riemann surface can be 
uniformized by at least one noded Schottky group. However, for our purposes, we 
will need to carefully choose our noded Schottky group.

%%%%%%%%%%%%%%%%%%%%%%%
\subsection{Explicit conformal specifications} We make some additional 
requirements on the explicit conformal structure on $S_0$; we require that $S_0$ 
admit a conformal automorphism $\phi$ of degree $3$, where $\phi(R_1)=R_2$, 
$\phi(R_2)=R_3$, $\phi(R_3)=R_1$, $\phi(G_1)=G_2$, $\phi(G_2)=G_3$ and 
$\phi(G_3)=G_1$. We also require that $S_0$ admit a reflection, $\psi$, with the 
properties that the three dividing geodesics, $R_1$, $R_2$, $R_3$ are invariant, 
but not pointwise fixed, under $\psi$, and the three non-dividing geodesics, 
$G_1$, $G_2$, $G_3$, are pointwise fixed under $\psi$.
There is a two real parameter family of such surfaces, we choose $S_0$ to be 
any surface within this family.

%%%%%%%%%%%%%%%%%%%%%%%
\subsection{Explicit topological specifications} We next need to choose a 
particular set of loops to define a Schottky group. We choose the loops, $W_1$, 
$W_2$, $W_3$, to have the following properties (see figure \ref{figure3}).

\begin{figure}
\centering
\includegraphics[width=8cm]{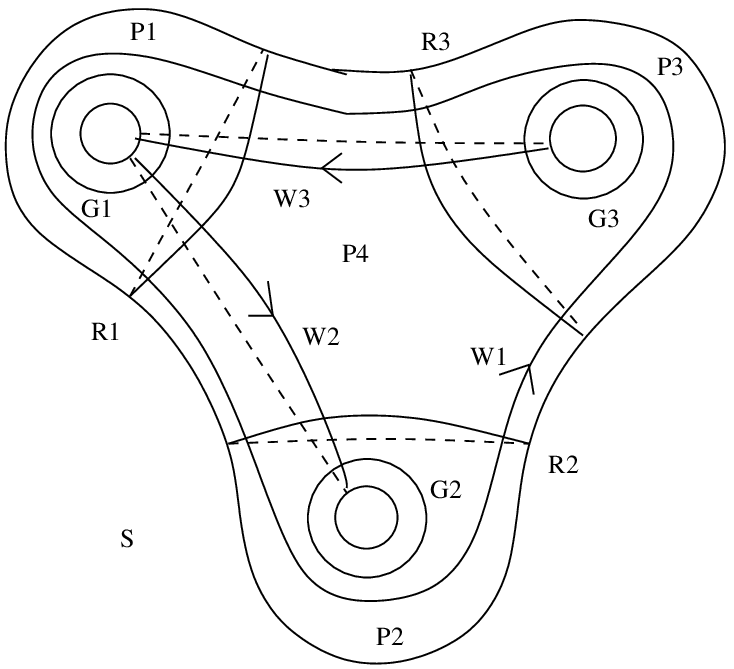}
\caption{}
\label{figure3}
\end{figure}

The loop $W_1$
crosses each of $R_1$, $R_2$ and $R_3$ exactly twice; is disjoint from each of 
$G_1$, $G_2$ and $G_3$; is invariant under the rotation $\phi$; and is pointwise 
fixed by the reflection $\psi$.
The loop $W_2$ crosses each of $R_1$ and $R_2$ exactly twice; is disjoint from 
$R_3$; crosses each of $G_1$ and $G_2$ exactly once; is disjoint from $G_3$; 
and is invariant under the reflection, $\psi$.
The loop $W_3=\phi\ngn (W_2)$.
One easily sees that $\phi(W_2)$ is freely homotopic to the product 
$W_3^{-1}\cdot W_2^{-1}$, from which it follows that the conformal map $\phi$ 
lifts to the Schottky group $G_0$ defined by these three loops. Since $\psi$ 
preserves each of $W_1$, $W_2$, $W_3$, it also lifts to an action on 
$\Omega(G_0)$.
It follows from well known properties of Schottky groups that the lift $\tilde 
\phi$ of $\phi$ is a M\"obius transformation, and that the lift $\tilde\psi$ of 
$\psi$ is an orientation-reversing M\"obius transformation with a Euclidean 
circle of fixed points.

\begin{Prop} For the Schottky group $G_0$, the set of loops, 
$R_1,R_2,R_3,G_1,G_2,G_3$ is pinchable. 
\end{Prop} 
\begin{proof} 
We need to show that none of these loops lifts to a loop, and that no two of 
them define conjugate elements of $G_0$.
Since $G_1$ crosses $W_2$ exactly once, it does not lift to a loop; likewise 
$G_2$ and $G_3$ do not lift to loops.
If say $R_2$ were to lift to a loop, then by cutting and pasting with $W_2$, we 
would obtain a loop that lifts to a loop and that crosses $W_1$ exactly once, 
which cannot be. Hence $R_1$, $R_2$ and $R_3$ do not lift to loops.
If two of these loops were to define the same element of $G_0$, then they would 
have the same crossing pattern with the triple of loops, $W_1,W_2,W_3$. But in 
fact, if one looks at the $6\times 3$ matrix of geometric crossings of these 
loops with the $W_i$, no two columns are the same; hence no pair of these loops 
can define the same or conjugate elements of $G_0$. 
\end{proof}

We have shown that the loops $R_1,R_2,R_3,G_1,G_2,G _3$ are pinchable; hence we 
can define the noded Schottky group $\widehat G_0$ obtained by pinching these 
loops. We require that the pinching be done so that the maps $\tilde \phi$ and 
$\tilde \psi$ are defined and normalize the Schottky group at each step of the 
pinching; hence they are defined and normalize $\widehat G_0$. On the noded 
Riemann surface $\widehat S_0$, obtained by pinching these six loops, the loops 
$\widehat W_1$, $\widehat W_2$ and $\widehat W_3$ are well defined.
We remark that, as a consequence of the fact that $\tilde\psi$ lifts to a 
homeomorphism of $\Omega^+(\widehat G_0)$, we have that the lift of $\tilde 
\psi$ is a fractional linear reflection, from which it follows that its fixed 
points set is a Euclidean circle. Since $\widehat W_1$ is pointwise fixed under 
$\tilde \psi$, it follows that the projection of this circle of fixed points, 
including the parabolic fixed points it passes through, is exactly $\widehat 
W_1$.

%%%%%%
\subsection{} 
It is easy to observe that the loops $W_2$ and $W_3$ are parallel going from 
$R_1$ to $G_1$ and back to $R_1$; hence there are lifts of $\widehat W_2$ and 
$\widehat W_3$, that are parallel between three successive parabolic fixed 
points.
We need to check that this same phenomenon occurs for every choice of defining 
loops. In what follows, we assume that we are given some set of three defining 
loops, $B_1$, $B_2$ and $B_3$, for our particular Schottky group, which we refer 
to as {\em black loops}. We assume that it does not happen that there are two 
arcs of black loops that are parallel between three successive crossings of red 
and green loops.
In what follows, we denote the homology intersection number of the loops, $A$ 
and $B$ by $A\times B$, and we denote the geometric intersection number by 
$A\bullet B$.

\begin{Prop} Let $\tilde S\to S$ be a regular covering of a surface, where 
$\tilde S $ is planar, and let $A$ and $B$ be loops on $S$ that lift to loops on 
$\tilde S$. Then $A\times B=0$. 
\end{Prop} 
\begin{proof} 
We can assume without loss of generality that $A$ and $B$ cross a finite number 
of times. Let $\tilde A$ be any lift of $A$, and let $\tilde B_1,\ldots,\tilde 
B_k$ be the set of lifts of $B$ that cross $\tilde A$. Since $\tilde S$ is 
planar, for each $i$, $\tilde A\times \tilde B_i=0$. The result follows from the 
fact that $A\times B=\sum\tilde A\times \tilde B_i$. 
\end{proof}

\begin{Cor}\label{Cor:xprod} If $B$ is a black loop, then $\sum_{i=1}^3 B\times 
G_i=0$. 
\end{Cor} 
\begin{proof} Observe that, after appropriate orientation, 
$W_1$ is homologous to $G_1+G_2+G_3$. Since $B$ and $W_1$ both lift to loops, 
$\sum B\times G_i=B\times W_1=0$. 
\end{proof}

\begin{Prop} No $B_j$ lies entirely within any one of the $P_i$; i.e., every 
black loop crosses at least one red loop. 
\end{Prop} 
\begin{proof} It is essentially obvious that no $B_i$ can lie entirely within 
the planar region, $P_4$, for every simple homotopically non-trivial loop in 
$P_4$ is parallel to the boundary; i.e., it is freely homotopic to a red loop.
If some $B_j$ were to lie in one of the tori, $P_1$, $P_2$ or $P_3$, then the 
commutator of this loop with a conjugate loop would be freely homotopic to the 
red loop on the boundary, which cannot be, as the red loop does not lift to a 
loop. 
\end{proof}

In what follows, the subsurfaces, $P_1$, $P_2$, $P_3$, are called the {\em 
tori}, and the subsurface $P_4$ is called the {\em planar region}. Also, the 
arcs of black loops in each of the tori, are called {\em strands}. The 
proposition above shows that each black loop contains at least one strand.
For some purposes, we will replace the red boundary of a torus by a point, so 
that we can think of the torus as closed. If two strands cross at this point, 
then the endpoints of one of them on the red boundary separates the endpoints of 
the other; in this case, we say that the strands {\em cross} on the red 
boundary. Otherwise, we say that they are {\em parallel}. Of course, two simple 
homotopically non-trivial loops on a torus that meet at exactly one point either 
cross at that point or are homotopic (up to orientation).

\begin{Prop} Every $R_i$ is crossed by at least one $B_j$, and every $G_i$ is 
crossed by at least one $B_j$. 
\end{Prop} 
\begin{proof} Every loop that is not crossed by some $B_j$ lifts to a loop, and 
we require that the $R_i$ and the $G_i$ do not lift to loops. 
\end{proof}

\begin{Prop} Every black loop contains at least two strands. 
\end{Prop} 
\begin{proof} We know from the above that every black loop contains at least 
one strand. Suppose we had that say $B_1$ crosses only $R_1$, at exactly two 
points. Then $B_1$ consists of exactly a strand in $P_1$, and an arc in the 
planar region. Since the arc in $P_4$ is homotopically non-trivial and not 
parallel to the boundary of $P_4$, it separates $P_2$ from $P_3$. It follows 
that $\phi(B_1)$ crosses $B_1$ in $P_4$, and only in $P_4$. Appropriately 
cutting and pasting these two loops, we obtain a new simple homotopically 
non-trivial loop that also lifts to a loop, and that lies entirely in $P_4$, 
which cannot be.
\end{proof}

\begin{Prop} 
Every torus contains at least two strands. 
\end{Prop} 
\begin{proof} 
Since each $G_i$ is crossed by at least one black loop, each torus contains at 
least one strand.
If say $P_1$ were to contain only one strand, then we could find a simple 
homotopically non-trivial loop in $P_1$ that does not cross that strand, and 
hence does not cross any black loop. Such a loop necessarily lifts to a loop, 
and its commutator with a conjugate loop also lifts to a loop. But this 
commutator is freely homotopic to $R_1$, which does not lift to a loop. 
\end{proof}

\begin{Prop} If $A_1\ne A_2$ are distinct strands in the torus 
$P_i$, where $A_1\bullet G_i\ne A_2\bullet G_i$, then $A_1$ and $A_2$ cross on 
$R_i$. 
\end{Prop} 
\begin{proof} This is immediate from the fact that, on the 
closed torus, $A_1$ and $A_2$, are homotopically non-trivial and distinct. 
\end{proof}

\begin{Prop} If there are exactly two strands, $A_1$ and $A_2$, in 
some $P_i$, then $A_1\bullet G_i\ne A_2\bullet G_i$. Further, either both $A_1$ 
and $A_2$ cross $G_i$, or, after appropriate labeling, $A_1\bullet G_i=1$ and 
$A_2\bullet G_i=0$. 
\end{Prop} 
\begin{proof} Since $G_i$ does not lift to a 
loop, one of the strands in $P_i$ must cross $G_i$; we can assume that it is 
$A_1$. If $A_2$ also crosses $G_i$, then we cannot have that $A_1\bullet 
G_i=A_2\bullet G_i$, for then, after perhaps reorienting one of them, we would 
have that $A_1$ and $A_2$ are parallel through a red loop, at least one green 
loop, and then a red loop. One also easily sees that if $A_1\bullet G_i>1$, 
then $A_2\bullet G_i\ne 0$. 
\end{proof}

\begin{Prop} Let $A$ be a strand in the 
torus $P_i$. Then $A\bullet G_i\le 3$. 
\end{Prop} 
\begin{proof} If $A\bullet 
G_i\ge 4$, then there would be two distinct lifts of $A$, starting at some lift 
of $G_i$, that would be parallel between two successive lifts of $G_i$, followed 
by a lift of $R_i$, contradicting our basic assumption that $\widehat G_0$ is 
not sufficiently complicated. 
\end{proof}

\begin{Prop} There are at most three strands inside any torus. Further, if there 
are three strands in some $P_i$, then two of them are parallel without crossing 
$G_i$, and the third crosses $G_i$ exactly once. 
\end{Prop} 
\begin{proof} 
Suppose there are $K>1$ strands, $A_1,\ldots, A_K$, in $P_i$, where, for 
$j=1,\ldots,K-1$, $A_j\bullet G_i\ge A_{j+1}\bullet G_i$. We have already seen 
that $A_1\bullet G_i>0$. We cannot have two parallel strands crossing $G_i$ the 
same positive number of times, hence $A_j\bullet G_i> A_{j+1}\bullet G_i$, 
unless $A_j\bullet G_i=0$. In particular, $A_2\bullet G_i<A_1\bullet G_i$.
We first assume that $A_2\bullet G_i>0$, then $A_1\bullet G_i>1$, from which it 
follows that every strand in $P_i$ crosses $G_i$.
Continuing with the assumption that $A_2\bullet G_i>0$, we next observe that we 
cannot have that $A_1\bullet G_i=3$, and $A_2\bullet G_i=2$, for lifts of $A_1$ 
and $A_2$ would then be parallel through two successive lifts of $G_i$, followed 
by a lift of $R_i$. We conclude that if $A_2\bullet G_i>0$, then $K=2$, 
$A_2\bullet G_i=1$, and either $A_1\bullet G_i=3$, or $A_1\bullet G_i=2$.
We next take up the case that $A_2\bullet G_i=0$. Then, since the lift of $G_i$ 
must be primitive, $A_1\bullet G_i=1$. If $K\ge 3$, then the strands 
$A_2,\dots,A_K$ all do not cross $G_i$, and are all parallel. Since any pair of 
these parallel strands must come from distinct red loops, other than $R_i$, and 
must go to distinct red loops, there can be at most two of these strands. We 
have shown that $K\le 3$. 
\end{proof}

\begin{Prop} At least one of the $P_i$ has 
three strands in it. 
\end{Prop} 
\begin{proof} If not, then all three of the tori 
would contain exactly two strands. Since each black loop must contain at least 
two strands, it follows that each of the black loops contains exactly two 
strands. Then each of the two ends of one of the strands from $P_1$ connect to 
each of the two ends of one of the strands of $P_2$, and each of the two ends of 
the other strand in $P_1$ connects to each of the two ends of one of the strands 
of $P_3$. Since the endpoints of the first strand separate the endpoints of the 
second strand on $R_i$, these two black loops must intersect in $P_4$, which 
cannot be. 
\end{proof}

We assume, from here on, that there are three strands, $A_1$, $A_2$, $A_3$, of 
black loops in $P_1$, where $A_1\bullet G_1=1$, and $A_2$ and $A_3$ are 
parallel with $A_2\bullet G_1=A_3\bullet G_1=0$.

\begin{Prop} Of the six endpoints of strands on $R_1$, three connect to $R_2$ 
and three connect to $R_3$. Also, one end of $A_1$ connects with $R_2$ and the 
other end connects with $R_3$. 
\end{Prop} 
\begin{proof} 
We orient $A_1$, $A_2$ and $A_3$ so that $A_2$ and $A_3$ are parallel as 
oriented curves. Then $A_2$ and $A_3$ cannot both come from the same red loop 
other than $R_1$, and they cannot both go to the same red loop other than $R_1$. 
Hence we can assume that $A_2$ comes from $R_2$ and that $A_3$ comes from $R_3$. 
Then either $A_2$ also goes to $R_2$ and $A_3$ also goes to $R_3$, or $A_2$ goes 
to $R_3$ and $A_3$ goes to $R_2$. In either case, since $A_1$ crosses $A_2$ and 
$A_3$ on $R_1$, both endpoints of $A_1$ cannot go to $R_2$, and both endpoints 
of $A_1$ cannot go to $R_3$. Hence one end of $A_1$ connects with $R_2$ and one 
end connects with $R_3$. 
\end{proof}

\begin{Prop} $P_2$ and $P_3$ each contain exactly two strands. 
\end{Prop} 
\begin{proof} 
Suppose there were three strands in $P_2$, call them $A'_1$, $A'_2$ and $A'_3$, 
labelled as above, so that $A'_1\bullet G_2=1$, and $A'_2\bullet G_2=A'_3\bullet 
G_3=0$, and oriented as above, so that $A'_2$ and $A'_3$ are parallel as 
oriented curves.
Then three of the endpoints of strands in $P_1$ connect with three of the 
strands in $P_3$, and three of the endpoints of strands of $P_2$ connect with 
strands of $P_3$. This implies incidentally that $P_3$ also contains three 
strands. In any case, $A_1$ is the middle strand from $P_1$ connecting to $P_2$, 
and $A'_1$ is the middle strand from $P_2$ connecting to $P_1$. We have shown 
that there is exactly one black loop, call it $B_1$, containing the strands 
$A_1$ and $A'_1$. We now have that $B_1\bullet G_1=B_1\bullet G_2=1$, and that 
$B_2$ and $B_3$ both do not cross either $G_1$ or $G_2$. This implies that $G_1$ 
and $G_2$ define conjugate elements of the Schottky group $G$, which they do 
not. 
\end{proof}

We now observe that we have reached our final contradiction. For we have shown 
that, of the three black loops, one of them crosses $G_1$, and has three 
strands; the second does not cross $G_1$ and has two strands, the second strand 
lying in $P_2$, while the third black loop has two strands, one lying in $P_1$ 
and not crossing $G_1$, while the other lies in $P_3$. The second black loop 
cannot cross $G_2$, for that would contradict Corollary \ref{Cor:xprod}; 
likewise, the third black loop cannot cross $G_3$. That leaves us with the first 
black loop crossing each of $G_1$, $G_2$ and $G_3$, and crossing each of them 
once, which also contradicts Corollary \ref{Cor:xprod}.
\bibliography{bm} 
\bibliographystyle{plain} 
\end{document}